\documentclass{amsart}

\usepackage{amssymb}
\usepackage{amsmath}
\usepackage{amsthm}

\usepackage{mathtools}
\usepackage{enumitem}

\usepackage{todonotes}
\usepackage{color,xcolor}

\newtheorem{theorem}{Theorem}[section]
\newtheorem{lemma}[theorem]{Lemma}
\newtheorem{corollary}[theorem]{Corollary}
\newtheorem{proposition}[theorem]{Proposition}

\theoremstyle{definition}
\newtheorem{definition}[theorem]{Definition}

\theoremstyle{remark}
\newtheorem{remark}[theorem]{Remark}

\numberwithin{equation}{section}

\newcommand{\N}{\mathbb{N}}

\newcommand{\R}{\mathbb{R}}

\newcommand{\Rn}{\R^n}

\newcommand{\tr}{\operatorname{tr}}

\newcommand{\Laplace}{\Delta}

\newcommand{\weakto}{\rightharpoonup}
\newcommand{\weakstarto}{\overset{*}{\weakto}}

\begin{document}
	\title{Pinning of interfaces by localized dry friction}

\author{Luca Courte}
	\address[Luca Courte]{Abteilung f\"ur Angewandte Mathematik,
		Albert-Ludwigs-Universit\"at Freiburg, Raum 228, Hermann-Herder-Straße 10, 79104 Freiburg i. Br.}
	\email{luca.courte@mathematik.uni-freiburg.de}
	\urladdr{https://aam.uni-freiburg.de/mitarb/courte/index.html}

	\author{Patrick Dondl}
	\address[Patrick Dondl]{Abteilung f\"ur Angewandte Mathematik,
		Albert-Ludwigs-Universit\"at Freiburg, Raum 217, Hermann-Herder-Straße 10, 79104 Freiburg i. Br.}
	\email{patrick.dondl@mathematik.uni-freiburg.de}
	\urladdr{https://aam.uni-freiburg.de/agdo/index.html}
	
	\author{Ulisse Stefanelli}
	\address[Ulisse Stefanelli]{Faculty of Mathematics, University of Vienna, 
		Oskar-Morgenstern-Platz 1, 1090 Wien, Austria  and  Istituto di Matematica
		Applicata e Tecnologie Informatiche \textit{{E. Magenes}}, v. Ferrata 1, 27100
		Pavia, Italy.}
	\email{ulisse.stefanelli@univie.ac.at}
	\urladdr{http://www.mat.univie.ac.at/$\sim$stefanelli}

	
	
	\date{\today}
	
	\begin{abstract}
		We consider a model for the evolution of an interface in a heterogeneous environment governed by a parabolic equation. The heterogeneity is  introduced as obstacles exerting a localized dry friction. Our main result establishes the emergence of a rate-independent hysteresis for suitable randomly distributed obstacles, i.e., interfaces are pinned by the obstacles until a certain critical applied driving force is exceeded. The treatment of such a model in the context of pinning and depinning requires a comparison principle. We prove this property and hence the existence of viscosity solutions. Moreover, under reasonable assumptions, we show that viscosity solutions are equivalent to weak solutions.
	\end{abstract}
	
	\maketitle
	
	\section{Introduction}\label{sec:intro}
The question of whether interfaces, like phase- or domain boundaries, become stuck  (are {\it pinned}) or propagate freely in environments where obstacles, e.g., stemming from impurities in the medium, are present, is relevant for understanding the behavior of a large number of physical systems.  
Examples in this direction can be found in ferroelectricity, dislocation dynamics, solid-liquid phase change, and martensitic transformations, just to mention a few. 

In this article, we consider a model for interface evolution in a heterogeneous medium of the form
	\begin{equation} \label{eq:main}
		u_t  - \Laplace u + \varphi(x, u(x)) \partial R(u_t) \ni f(x, t).
	\end{equation}
 In this model, the graph of the function $u\colon \R^n\times[0,\infty) \to \R$, $n\ge 1$ represents an interface in the domain, e.g., a phase boundary. This phase boundary propagates in $\R^{n+1}$ according to a local driving force, which is comprised of a regularizing term (in this case the Laplacian as the linearization of a line tension) and a given driving force $f$ (which may come from, e.g., an external applied electric field in the case of a ferroelectric phase boundary). Additionally, the effect of the heterogeneities in the medium is modeled by the {\it dry-friction} term $-\varphi(x, u(x)) \partial R(u_t) $, with a space- and state-dependent weight $\varphi\ge 0$. The regions where $\varphi>0$ correspond to obstacles in the medium and activate a friction force. This is specified via the corresponding (pseudo)-potential of dissipation $R(\cdot) = |\cdot|$ and we indicate by $\partial$ its set-valued subdifferential, namely $\partial R(x) = x/|x|$ for $x \not =0$ and $\partial R(0)=[-1,1]$.

In \cite{Dondl_11b} an earlier, semilinear version of this model was discussed. In particular,    the equation 
\begin{equation}\label{semi}
u_t  - \Laplace u + \varphi(x,u(x)) =f 
\end{equation}
is addressed, see \cite{Dondl_11b} for further references regarding modeling. For this last class of models, the case of periodic heterogeneities was considered in \cite{Dirr:2006ui,Dirr:2008ki,Dondl_16b}. For the random case, see for example \cite{Dondl_11b,Coville:2009uv,Dondl_12a,Dondl_16e,Dondl:2015kk, Bodineau:2013ur, 1903.04952} where various aspects of the problem, including questions of pinning, depinning, and ballistic propagation of the interfaces, were considered and more references to the physics and engineering literature can be found.
In the case where $\varphi$ represents localized, according to a Poisson point process randomly distributed obstacles, it was shown \cite{Dondl_11b}---barring some degeneracies and dependencies--- that a stationary supersolution to equation \eqref{semi} exists. The form of $\varphi$ in this case is 
\[
\varphi(x) = \sum_{k\in\N} f_k(\omega)\varphi_0(x-x_k(\omega),y-y_k(\omega)),
\]
where $f_k(\omega) \ge 0$ are random obstacle strengths and $(x_k,y_k)$ is an $n+1$-dimensional Poisson point process. The obstacle shape is then given by the function $\varphi_0 \ge 0$, $\varphi_0 \in C_c^\infty(\R^{n+1})$. The setting of \eqref{semi} is however questionable from the modeling viewpoint, for the effect of obstacles is directional: they exert a driving force on the interface by pushing it downward. This is consistent with the case of an increasing $u$ only. It leads to the necessity of introducing some unnatural assumptions, e.g., in order to prove that a stationary solution exists one has to exclude the presence of obstacles crossing the line $u=0$. Otherwise, the force given by the function $\varphi$ may push the interface towards $-\infty$. Of course, this is not a physically reasonable situation. 

In order to amend such physical inconsistency, we focus here on 
the localized dry-friction dynamic of \eqref{eq:main}. This provides a sounder description of the interaction of obstacles and interfaces, which is in particular independent on the direction at which defects are traversed. In fact,  the friction force $-\varphi(x, u(x)) \partial R(u_t)$ is aligned with $-u_t$, so that the corresponding  contribution $\varphi(x, u(x)) \partial R(u_t) u_t = \varphi(x, u(x))R(u_t) \geq0$ to energy dissipation is always non negative, regardless of the sign of $u_t$ \cite{Moreau70}.

Dry-friction effects are ubiquitous in mechanics and have attracted attention since the pioneering observations by Coulomb. The reader is referred to the monographs \cite{Duvaut-Lions, Fremond02, KR, Visintin94} for a mathematically-oriented collection of classical materials on dry-friction modeling and analysis. In the context of dislocations dynamics, a localized dry-friction assumption similar to \eqref{eq:main} is advanced by Koslowski, Cuiti\~no, and Ortiz \cite{Koslowski:2002ur}. 

In the purely rate-independent case, i.e., when $u_t$ is omitted in \eqref{eq:main} and $\varphi$ is taken to be constant, existence of strong solutions to \eqref{eq:main} is readily checked by time discretization. The reader is referred to the monograph by Mielke and Roub\' \i\v cek~\cite{MR} for a thorough presentation of existence and approximation theories for rate-independent systems. In absence of the viscous term $u_t$ and for uniformly positive weights $\varphi$, equation \eqref{eq:main} would fall into the class of rate-independent systems with state-dependent dissipation, which admit a well-posedness theory \cite{BKS, Kunze, MR07}.

The case of \eqref{eq:main} is more involved, for mixed rate-independent and rate-dependent dynamics occur. 
Existence for such mixed systems is also available. The reader is referred to the series of contributions by Mielke, Rossi, and Savar\'e \cite{MRS1, MRS2, MRS3} on these topics. The available theory seems however not to cover the specific case of \eqref{eq:main}, where the switching between purely viscous and mixed dynamics is driven by the state itself.

In order to treat the model in \eqref{eq:main} in the context of pinning and depinning of interfaces, some technical hurdles have to be overcome. In particular, we abstain from following the by-now classical path of the variational theory of rate-independent processes \cite{MR}, or the theory of balanced viscosity solutions \cite{1712.06851, knees_rossi_zanini_2019, Efendiev_Mielke}, and focus instead on solvability in the viscosity sense \cite{Crandall1992}.  It should be noted that viscosity solutions and balanced viscosity solutions (BV-solutions) are two different concepts.  The theory of viscosity solutions is instrumental for establishing comparison tools which are crucial in order to give meaning to the eventually constructed super- and subsolutions. The existence of viscosity solutions and the comparison principle are proved in Section~\ref{sec:vis_sols} for some more general problem including \eqref{eq:main}. 

We then return to our specific choice \eqref{eq:main} in Section  \ref{sec:application}. In Section \ref{ssec:equivalence} we examine the relation of viscosity solutions to the already well known weak solutions for \eqref{eq:main} on the flat torus. The pinning result is then presented in Section \ref{ssec:pinning}. In the same setting of randomly distributed, localized obstacles as treated in \cite{Dondl_11b}, we show existence of stationary super- and subsolutions for $-F^* \le f \le F^*$ for some deterministic $F^*>0$. This fully establishes the emergence of rate independent hysteresis in such a system when considering (quasistatic) time-dependent loading with changing sign. In particular, we see that the localized dry friction may be used like an additional driving force with the appropriate sign (i.e., negative when constructing a supersolution, positive when constructing a subsolution). A precise statement of our main result is given in Theorem \ref{thm:main}.
	
	\section{Viscosity Solutions for Problems with Rate-Independent Dissipation}\label{sec:vis_sols}
	In order to treat the differential inclusion \eqref{eq:main} (referred to as {\it equation} in the following) we are going to resort to the notion of viscosity solutions (see for instance \cite{Crandall1992}). In fact we will focus our attention on the following more general  equation 
	\begin{equation} \label{eq:pde}
		u_t + F(x, t, u, \nabla u, D^2 u) \in \mathcal{S}(u_t) G(x, t, u, \nabla u) \text{ in } \Omega \times I,
	\end{equation}
	with $\Omega \subset \Rn$ open and connected, $I \coloneqq (0, T)$ for some final reference time $T>0$, $F: \Omega \times I \times \R \times \Rn \times \operatorname{Sym}(n) \to \R$, $G: \Omega \times I \times \R \times \Rn \to [0,\infty)$, $\mathcal{S} : \R \to \mathcal{P}(\R)$, where $\operatorname{Sym}(n)$ is the set of all symmetric $n\times n$ matrices, and $\mathcal{P}(\R)$ is the power set of $\R$. The specific case of equation \eqref{eq:main} is recovered by choosing $$F(x, t, r, p, X) \coloneqq -\tr(D^2 u) - f(x, t), \quad G(x, t, r, p) \coloneqq \varphi(x, r), \quad \mathcal{S}(a) \coloneqq -\partial R(a).$$ In particular, recall that one has $\mathcal{S}(0) = [-1, 1]$ and $\mathcal{S}(a)=-a/|a|$ if $a \not =0$.

We denote the {\it parabolic superjet (subjet)} of the function $u : \Omega \times I \to \R$ at the point $(x, t) \in \Omega \times (0, T)$ by  $\mathcal{P}^{2, +} u(x, t)$, $\mathcal{P}^{2, -} u(x, t)$, respectively, see \cite[Section~8]{Crandall1992} for a rigorous definition, we introduce the following definition.

	\begin{definition}[Viscosity Solutions]\label{def:vis_sol}
\it An upper-semicontinuous function \linebreak $u\in \mathrm{USC}(\Omega\times I)$ is a \emph{viscosity subsolution} of \eqref{eq:pde} if for all $(x, t) \in \Omega \times I$ and for all $(a, p, X) \in \mathcal{P}^{2, +} u(x, t)$ there exists $\mu \in \mathcal{S}(a)$ such that
		\[
			a + F(x, t, u(x, t), p, X) \le \mu G(x, t, u(x, t), p).
		\]
		A lower-semicontinuous function $v\in \mathrm{LSC}(\Omega\times I)$ is a \emph{viscosity supersolution} of \eqref{eq:pde} if for all $(y, s) \in \Omega \times I$ and for all $(b, q, Y) \in \mathcal{P}^{2, -} v(y, s)$ there exists $\nu \in \mathcal{S}(b)$ such that
		\[
			b + F(y, s, v(y, s), q, Y) \ge \nu G(y, s, v(y, s), q).
		\]
		A \emph{viscosity solution} of \eqref{eq:pde} is a function that is both viscosity sub- and supersolution.
	\end{definition}
	This notion generalizes the definition of viscosity solutions from \cite{Crandall1992} to the class of differential inclusions we are interested in. In order to prove that this generalization is indeed appropriate let us now check that, under rather general assumptions, uniform limits of viscosity solutions in the sense of \cite{Crandall1992} are viscosity solutions in the sense of Definition \ref{def:vis_sol}.
	
	\begin{proposition}\label{thm:stability}
		Let $u_\epsilon \in \mathrm{USC}(\Omega \times I)$ be a sequence of viscosity subsolutions (in the sense of \cite{Crandall1992}) of the regularized problem
		\[
		u_t + F_\epsilon(x, t, u, \nabla u, D^2 u) = S_\epsilon(u_t) G_\epsilon(x, t, u, \nabla u),
		\]
		with $F_\epsilon : \Omega \times I \times \R \times \Rn \times \operatorname{Sym}(n) \to \R$ continuous, $G_\epsilon : \Omega \times \R \times \R \times \Rn \to \R$ continuous,  and $S_\epsilon : \R \to \R$.
		
Assume that $u_\epsilon \to u$, $F_\epsilon \to F$, and $G_\epsilon \to G$ locally uniformly and that whenever $a_\epsilon \to a$ there exists a not relabeled convergent subsequence $S_\epsilon(a_\epsilon) \to \mu \in \mathcal{S}(a)$.  Then, $u$ is a viscosity subsolution of \eqref{eq:pde} in the sense of Definition \emph{\ref{def:vis_sol}}.
	\end{proposition}
	\begin{proof}
	    The local uniform convergence of $u_\epsilon \to u$ standardly implies that, for every $(x, t) \in \Omega \times I$ and every $(a, p, X) \in \mathcal{P}^{2, +} u(x, t)$, we may find a sequence $(x_\epsilon, t_\epsilon) \in \Omega \times I$ and $(a_\epsilon, p_\epsilon, X_\epsilon) \in \mathcal{P}^{2, +} u_\epsilon(x_\epsilon, t_\epsilon)$ such that
		\[
			(x_\epsilon, t_\epsilon, u_\epsilon(x_\epsilon, t_\epsilon), a_\epsilon, p_\epsilon, X_\epsilon) \to (x, t, u(x, t), a, p, X)
		\]
		as $\epsilon \to 0$ \cite[Proposition~4.3]{Crandall1992}. By passing to a subsequence, our assumptions provide us with a $\mu \in \mathcal{S}(a)$ such that
		\[
			S_\epsilon(a_\epsilon) \to \mu.
		\]
		As $u_\epsilon$ are viscosity subsolutions in the sense of \cite{Crandall1992}, it holds
		\[
			a_\epsilon + F_\epsilon(x_\epsilon, t_\epsilon, u_\epsilon(x_\epsilon, t_\epsilon), p_\epsilon, X_\epsilon) \le S_\epsilon(a_\epsilon) G_\epsilon(x_\epsilon, t_\epsilon, u_\epsilon(x_\epsilon, t_\epsilon), p_\epsilon).
		\]
		Using the local uniform convergence of $F_\epsilon, G_\epsilon$, we can pass to the limit and obtain
		\[
			F(x, t, u(x, t), a, p, X) \le \mu G(x, t, u(x, t), p),
		\]
		and $u$ is a viscosity subsolution \eqref{eq:pde} in the sense of Definition {\ref{def:vis_sol}}.
	\end{proof}
		
	\subsection{Comparison Principles}
\label{ssec:comp_existence}
	In this section, we will prove first a comparison principle on bounded domains (Theorem \ref{thm:comp}) and then in $\Rn$ (Theorem \ref{thm:comp_rn}). The exposition of this section and the proofs follow the structure of \cite[Sections~3, 5.D., and~8]{Crandall1992}.

Let us start by listing assumptions on the nonlinearities $F$, $G$, and $\mathcal{S}$ in \eqref{eq:pde}. We ask the following:
 
 The function
  $F : \Omega \times I \times \R \times \Rn \times \mathrm{Sym}(n) \to
  \R$ is continuous and it holds
  \begin{enumerate}[label=F\arabic*)]
  \item\label{it:F_inc} $F(x, t, r, p, X) \le F(x, t, s, p, X)$ for
    all $r, s \in \R$ with $r \le s$ ,
  \item\label{it:F_modulus} There exists a modulus of continuity
    $\omega_F$ such that
    \[
      F(y, t, r, \alpha (x-y), Y) - F(x, t, r, \alpha (x-y), X) \le
      \omega_F(|x-y| + \alpha |x-y|^2)
    \]
    whenever $x, y \in \Omega, t \in I, r\in \R, \alpha \in \R_{\ge0}$
    and
    \[
      -3\alpha \left(\begin{array}{cc}\operatorname{Id} & 0 \\ 0 &
          \operatorname{Id}\end{array}\right) \le
      \left(\begin{array}{cc}X & 0 \\ 0 & -Y\end{array}\right)\le
      3\alpha \left(\begin{array}{cc}\operatorname{Id} &
          -\operatorname{Id} \\ -\operatorname{Id} &
          \operatorname{Id}\end{array}\right).\]
  \end{enumerate}
 
	The function $G :  \Omega \times I \times \R \times \Rn \to [0, \infty)$ is continuous and satisfies
	\begin{enumerate}[label=G\arabic*)]
		\item\label{it:G_Lipschitz} $\left|G(x, t, r, p)-G(x, t, s, p)\right| \le L_G |r-s|$ for some $L_G>0$ and all $r, s \in \R$,
		\item\label{it:G_modulus} There exists a modulus of continuity $\omega_G$ such that 
		\[
		G(y, t, r, \alpha (x-y)) - G(x, t, r, \alpha (x-y)) \le \omega_G(|x-y| + \alpha |x-y|^2),
		\]
		whenever $x, y \in \Omega, t \in I, r\in \R$.
	\end{enumerate}

The set-valued function $\mathcal{S} :\R \to \mathcal{P}(\R)\setminus \emptyset $ is such that 
	\begin{enumerate}[label=S\arabic*)]
		\item\label{it:S_mon} $-\mathcal{S}$ is {\it maximal monotone} \cite{Brezis73}, namely $\sup \mathcal S(a) \le \inf \mathcal S(b)$ for all $a, b \in \R$ with $a > b$, and the graph of $-\mathcal S$ cannot be properly extended (in the sense of graph inclusion) by a monotone graph,
		\item\label{it:S_bdd} the range $\mathcal{S}(\R) = \cup_{a\in \R} \mathcal S(a)$ is bounded, i.e., $ \mathcal{S}(\R) \subset [-\mathcal{S}_\mathrm{max},\mathcal{S}_\mathrm{max}]$ for some $\mathcal{S}_\mathrm{max}>0$.
	\end{enumerate}
 
Note that all assumptions are fulfilled in the specific case of \eqref{eq:main}.
	\begin{theorem}[Comparison Principle I]\label{thm:comp}
		Let $\Omega \subset\Rn$ be bounded. Assume \emph{\ref{it:F_inc}}, \emph{\ref{it:F_modulus}},  \emph{\ref{it:G_Lipschitz}}, \emph{\ref{it:G_modulus}}, and \emph{\ref{it:S_mon}}, \emph{\ref{it:S_bdd}}, and let $u$ be a viscosity subsolution and $v$ be a viscosity supersolution of \eqref{eq:pde} with $u \le v$ on the parabolic boundary $\partial_P (\Omega \times I)$, i.e., on $ \partial\Omega \times I \cup \Omega \times \{0\}$. Then, $u \le v$ in $\Omega \times I$.
	\end{theorem}
	\begin{proof} Comparison principles usually hinge on the monotonicity of the nonlinear terms in the equation.  
As $G$ is not necessarily monotone in $u$, we start by performing an exponential rescaling, i.e., we let $\lambda > 0$ and define $U(x, t) \coloneqq e^{-\lambda t} u(x, t)$ and $V(x, t) \coloneqq e^{-\lambda t} v(x, t)$. It is easy to see that $U$ and $V$ are sub- and supersolution to
		\[
		U_t + \lambda U + \tilde{F}(x, t, U, \nabla U, D^2 U) \in \mathcal{S}(e^{\lambda t}(U_t + \lambda U)) \tilde{G}(x, t, U, \nabla U)
		\]
		with $$\tilde{F}(x, t, r, p, X) \coloneqq e^{-\lambda t} F(x, t, e^{\lambda t}r, e^{\lambda t}p, e^{\lambda t}X)$$ and $$\tilde{G}(x, t, r, p) \coloneqq e^{-\lambda t} G(x, t, e^{\lambda t}r, e^{\lambda t}p).$$ The functions $\tilde{F}$ and $\tilde{G}$ satisfy \ref{it:F_inc}, \ref{it:F_modulus}, and \ref{it:G_Lipschitz}, \ref{it:G_modulus}, respectively, with the same constants by rescaling the moduli by $e^{-\lambda t}$.
		
		Let us check the comparison for $U$ and $V$. Assume by contradiction that comparison does not hold, i.e.,
		\[
		\sup_{\substack{x\in \Omega \\ t\in I}} \left\{ U(x, t) - V(x, t)\right\} \eqqcolon \delta > 0
		\]
		and define
		\[
		M_{\alpha, \gamma} \coloneqq \sup_{\substack{x, y \in \Omega \\ t\in I}} \left\{ U(x, t) - V(y, t) - \tfrac{\alpha}{2}|x-y|^2 - \tfrac{\gamma}{T-t}\right\}.
		\]
		We have $M_{\alpha, \gamma} > \delta/2$ for $\gamma$ small enough.
		Since the domain is bounded, the supremum is achieved at a point $(\hat{x}, \hat{y}, \hat{t}) \in \overline{\Omega}\times \overline{\Omega} \times [0,T)$.
		
		We will now show that the triplet $(\hat{x}, \hat{y}, \hat{t})$ is in the interior of the parabolic domain if $\alpha$ is large enough. Assume first that $\hat t = 0$ then
		\[
		M_{\alpha, \gamma } = U(\hat{x}, 0) - V(\hat{y}, 0) - \tfrac{\alpha}{2}|\hat x-\hat y|^2 - \tfrac{\gamma}{T-\hat t} \le  - \tfrac{\alpha}{2}|\hat x-\hat y|^2 - \tfrac{\gamma}{T-\hat t} \le 0,
		\]
		since $u\le v$ on the parabolic boundary $\partial_P (\Omega \times I)$, contradicting $M_{\alpha, \gamma} > \delta/2$. We now check that, if $\alpha$ is chosen to be large enough, $\hat{x}$ and $\hat{y}$ necessarily belong to $\Omega$. Assume the contrary, namely there exists a subsequence $\alpha_n \to \infty$ with $\hat{x}_n \in \partial \Omega$ realizing the sup. Then, we also have $\hat{y}_n \to \hat{y}_\infty \in \partial \Omega$ and therefore
		\[
		\lim_{n\to \infty}M_{\alpha_n, \gamma } = \lim_{n\to \infty}\left(U(\hat{x}_n, \hat{t}) - V(\hat{y}_n, \hat{t}) - \tfrac{\alpha_n}{2}|\hat x_n-\hat y_n|^2 - \tfrac{\gamma}{T-\hat t}\right) \le 0 - \tfrac{\gamma}{T} \le 0,
		\]
		where we used again that $u \le v$ on $\partial_P(\Omega\times I)$ and reached a contradiction. Therefore, we have proved that $(\hat{x}, \hat{y}, \hat{t}) \in \Omega \times \Omega \times (0, T)$, at least if $\alpha$ is large enough.  Hence, we have \cite[Theorem 8.3]{Crandall1992}
		\[
		(a, \alpha(\hat{x}-\hat{y}), X) \in \mathcal{P}^{2, +} U(\hat{x}, \hat{t}) \text{ and }
		(b, \alpha(\hat{x}-\hat{y}), Y) \in \mathcal{P}^{2, -} V(\hat{y}, \hat{t}) 
		\]
		with $a-b = \tfrac{\gamma}{(T-\hat{t})^2}$ and $$-3\alpha \left(\begin{array}{cc}\operatorname{Id} & 0 \\ 0 & \operatorname{Id}\end{array}\right) \le \left(\begin{array}{cc}X & 0 \\ 0 & -Y\end{array}\right)\le 3\alpha \left(\begin{array}{cc}\operatorname{Id} & -\operatorname{Id} \\ -\operatorname{Id} & \operatorname{Id}\end{array}\right),$$ which implies that $X \le Y$. As $U$ is a subsolution and $V$ is a supersolution, we can find $\mu \in \mathcal{S}(e^{\lambda t}(a+\lambda U))$ and $\nu \in \mathcal{S}(e^{\lambda t}(b+\lambda V))$ such that
		\begin{align}
		a + \lambda U + F(\hat x,\hat t, U, \alpha(\hat{x}-\hat{y}), X) - \mu G(\hat x,\hat  t, U, \alpha(\hat{x}-\hat{y})) \le 0,\label{eq:sub} \\
		b + \lambda V + F(\hat y,\hat t, V, \alpha(\hat{x}-\hat{y}), Y) - \nu G(\hat y,\hat  t, V, \alpha(\hat{x}-\hat{y})) \ge 0.\label{eq:super}
		\end{align}
		By subtracting \eqref{eq:sub} from \eqref{eq:super}, we obtain
		\begin{align*}
		a - b \le&~ \lambda (V-U)  + F(\hat y,\hat t, V, \alpha(\hat{x}-\hat{y}), Y) - F(\hat x,\hat t, U, \alpha(\hat{x}-\hat{y}), X) \\
		&- \nu G(\hat y,\hat  t, V, \alpha(\hat{x}-\hat{y})) + \mu G(\hat x,\hat  t, U, \alpha(\hat{x}-\hat{y})).
		\end{align*}
		Adding and subtracting terms, we get
		\begin{align*}
		\tfrac{\gamma}{(T-\hat{t})^2} \le&~ \lambda (V-U) + F(\hat y,\hat t, V, \alpha(\hat{x}-\hat{y}), Y) - F(\hat x,\hat t, V, \alpha(\hat{x}-\hat{y}), X) \\
		&+ F(\hat x,\hat t, V, \alpha(\hat{x}-\hat{y}), X)- F(\hat x,\hat t, U, \alpha(\hat{x}-\hat{y}), X) \\
		&- \nu G(\hat y,\hat  t, V, \alpha(\hat{x}-\hat{y}))+ \nu G(\hat x,\hat  t, V, \alpha(\hat{x}-\hat{y})) \\
		&- \nu G(\hat x,\hat  t, V, \alpha(\hat{x}-\hat{y}))+ \nu G(\hat x,\hat  t, U, \alpha(\hat{x}-\hat{y})) 
		\\&- \nu G(\hat x,\hat  t, U, \alpha(\hat{x}-\hat{y})) + \mu G(\hat x,\hat  t, U, \alpha(\hat{x}-\hat{y})) \\
		\le&~ \lambda (V-U) + e^{-\lambda \hat t}\omega_F(|\hat x-\hat y| + \alpha |\hat x-\hat y|^2) + 0 \\
		&+|\nu|  e^{-\lambda\hat t}\omega_G(|\hat x-\hat y|+ \alpha |\hat x-\hat y|^2) 
		+ |\nu| L_G |U-V| \\
		&+ (\mu - \nu) G(\hat x,\hat  t, U, \alpha(\hat{x}-\hat{y})).
		\end{align*}
		Where the second inequality follows from \ref{it:F_inc}, \ref{it:F_modulus} and \ref{it:G_Lipschitz}, \ref{it:G_modulus} by noticing that, $U > V$, $X \le Y$, and $a > b$. In particular, we have $e^{\lambda t}(a+\lambda U) > e^{\lambda t}(b + \lambda V)$ and \ref{it:S_mon} implies that $\mu-\nu \le 0$. As $G \ge 0$, it follows that the last term above is negative. Eventually, by means of \ref{it:S_bdd} we get
		\begin{align*}
		\tfrac{\gamma}{(T-\hat{t})^2} \le&~\lambda (V-U) + e^{-\lambda \hat t}\omega_F(|\hat x-\hat y|+ \alpha |\hat x-\hat y|^2) \\
		&+|\nu|  e^{-\lambda\hat t}\omega_G(|\hat x-\hat y|+ \alpha |\hat x-\hat y|^2) + |\nu| L_G |U-V| \\
		\le&~ (\mathcal{S}_\mathrm{max} L_G -\lambda) (U-V) \\
		& + (1+\mathcal{S}_\mathrm{max})\big(\omega_F(|\hat x-\hat y|+ \alpha |\hat x-\hat y|^2) \\
		&+ \omega_G(|\hat x-\hat y|+ \alpha |\hat x-\hat y|^2)\big).
		\end{align*}
		By choosing $\lambda \ge \mathcal{S}_\mathrm{max} L_G$ the first term becomes negative and hence
		\begin{align*}
		0 < \tfrac{\gamma}{T^2} &\le \tfrac{\gamma}{(T-\hat{t})^2} \\
		&\le (1+\mathcal{S}_\mathrm{max}) \left(\omega_F(|\hat x-\hat y|+ \alpha |\hat x-\hat y|^2) + \omega_G(|\hat x-\hat y|+ \alpha |\hat x-\hat y|^2)\right).
		\end{align*}
		By taking $\alpha \to \infty$, we have $\alpha |\hat x-\hat y|^2 \to 0$ which implies $|\hat x-\hat y| \to 0$. Therefore, the right-hand side above goes to $0$, leading to a contradiction.
	\end{proof}
	
	\begin{corollary}[Comparison Principle II]\label{cor:comp_Lipschitz}
		Let $\Omega \subset\Rn$ be bounded.
Assume \emph{\ref{it:F_modulus}},  \emph{\ref{it:G_Lipschitz}}, \emph{\ref{it:G_modulus}}, and \emph{\ref{it:S_mon}}, \emph{\ref{it:S_bdd}}.
Instead of \emph{\ref{it:F_inc}}, $F$ is asked to satisfy
		\begin{align*}
&		| F(x, t, r, p, X) - F(x, t, s, p, X) | \le L_F |r-s| \\
&\text{for some $L_F>0$ and all} \  r, s\in \R, \, p \in \R^n, \, X \in {\rm Sym}(n).
		\end{align*}
		Let $u$ be a viscosity subsolution and $v$ be a viscosity supersolution of \eqref{eq:pde} with $u \le v$ on $\partial_P (\Omega \times I)$. Then,  $u \le v \text{ in } \Omega \times I$.
	\end{corollary}
	\begin{proof}
		The result follows from the first comparison principle and an exponential scaling. Proceeding as in the proof of Theorem \ref{thm:comp}, we let $\lambda > 0$ and define $U(x, t) \coloneqq e^{-\lambda t} u(x, t)$ and $V(x, t) \coloneqq e^{-\lambda t} v(x, t)$. Recall that $U$ and $V$ are sub- and supersolution to
		\[
		W_t + \lambda W + e^{-\lambda t} F(x, t, e^{\lambda t}W, e^{\lambda t}\nabla W, e^{\lambda t}D^2W) \in \mathcal{S}(e^{\lambda t}(W_t + \lambda W)) \tilde{G}(x, t, W, \nabla W)
		\]
		with $\tilde{G}(x, t, r, p) \coloneqq e^{-\lambda t} G(x, t, e^{\lambda t}r, e^{\lambda t}p)$. The function $\tilde{G}$ satisfies \ref{it:G_Lipschitz} and \ref{it:G_modulus} with the same constant by rescaling the modulus by $e^{-\lambda t}$. On the other hand, the function $\tilde{F}(x, t, r, p, X) \coloneqq \lambda r + e^{-\lambda t} F(x, t, e^{\lambda t}r, e^{\lambda t}p, e^{\lambda t}X)$ clearly satisfies \ref{it:F_modulus}. Let $r \le s$ and compute
		\begin{align*}
		\tilde{F}(x, t, r, p, X) - \tilde{F}(x, t, s, p, X) =&~\lambda (r-s) + e^{-\lambda t} F(x, t, e^{\lambda t}r, e^{\lambda t}p, e^{\lambda t}X) \\
		&- e^{-\lambda t} F(x, t, e^{\lambda t}s, e^{\lambda t}p, e^{\lambda t}X)\\
		\le&~\lambda (r-s) + L_F |r-s| \\
		= &~(L_F -\lambda) |r-s|
		\end{align*}
		which shows that, choosing $\lambda \ge L_F$, $\tilde{F}$ fulfills \ref{it:F_inc} as well. We can now apply Theorem \ref{thm:comp} and obtain that $U \le V$, hence $u \le v$.
	\end{proof}

	\begin{remark}\label{rmk:torus}
		The comparison principles proved above would also hold on the flat torus $\mathbb{T}^n$, the main difference being that the parabolic boundary $\partial_P(\mathbb{T}^n \times I) = \mathbb{T}^n \times \{0\}$, i.e., we only have to specify an initial condition. To see that  comparison holds just note that we can assume $\mathbb{T}^n  = \R^n/{\mathbb Z}^n \sim [-1, 1]^n$. In the proof of the comparison principle a maximum of $M_{\alpha, \gamma}$ will always be achieved in the interior of some  larger domain say $[-2, 2]^n$, letting condition $u \le v$ on $\partial [-2, 2]^n \times I$ redundant.
	\end{remark}
	
	As it is usual in the treatment of viscosity solutions, assumptions on $F$ and $G$ have to be strengthened in order to be able to prove comparison results on the whole space $\Rn$. We hence replace F2) and G2) by the following conditions:  
\begin{enumerate}[label=FU)]
		\item\label{it:F_growth} The function $F$ can be written as $F(x, t, r, p, X) = F_1(x, t, r) + F_2(t, p, X)$ with $F_1$ and $F_2$ continuous. Moreover, there are $C_{F_1} > 0$, $K_F > 0$, and moduli of continuity $\omega_{F_1}$ and $\omega_{F_2}$, such that for all $x, y \in \Rn$, $t \in I$, $r\in \R$, $a, b \in \Rn$, $p, q \in \Rn$, and $X, Y \in \operatorname{Sym}(n)$, the following conditions hold:
		\begin{enumerate}[label=(\roman*)]
			\item $F_1(y, t, r) - F_1(x, t, r) \le \omega_{F_1}(|x-y|)$.
			\item $F_2(t, \alpha(x-y), Y) - F_2(t, \alpha(x-y), X) \le \omega_{F_2}(|x-y| + \alpha |x-y|^2)$, 
			whenever
			\[
			-3\alpha \left(\begin{array}{cc}\operatorname{Id} & 0 \\ 0 & \operatorname{Id}\end{array}\right) \le \left(\begin{array}{cc}X & 0 \\ 0 & -Y\end{array}\right)\le 3\alpha \left(\begin{array}{cc}\operatorname{Id} & -\operatorname{Id} \\ -\operatorname{Id} & \operatorname{Id}\end{array}\right).\]
			\item $F_1(y, t, r) - F_1(x, t, r) \le C_{F_1} + K_F |x-y|$.
		\end{enumerate}
		
	\end{enumerate}
	\begin{enumerate}[label=GU)]
		\item\label{it:G_growth} The function $G$ can be written as $G(x, t, r, p) = G_1(x, t, r) + G_2(t, p)$ with $G_1$ and $G_2$ continuous. Moreover, there are $C_{G_1} > 0$, 
		$K_G > 0$, and a modulus of continuity $\omega_{G_1}$, such that for all $x, y \in \Rn$, $t \in I$, $r\in \R$, and $p, q \in \Rn$, the following conditions hold:
		\begin{enumerate}[label=(\roman*)]
			\item $G_1(x, t, r) - G_1(y, t, r) \le \omega_{G_1}(|x-y|)$,
			\item $G_1(x, t, r) - G_1(y, t, r) \le C_{G_1} + K_G |x-y|$,
		\end{enumerate}
	\end{enumerate}
	
	\begin{theorem}[Comparison Principle in $\Rn$]\label{thm:comp_rn}
Assume \emph{\ref{it:F_inc}}, \emph{\ref{it:F_growth}}, \emph{\ref{it:G_Lipschitz}}, \emph{\ref{it:G_growth}}, and  \emph{\ref{it:S_mon}}, \emph{\ref{it:S_bdd}} on $\Rn$, and let $u$ be a viscosity subsolution and $v$ be a viscosity supersolution of \eqref{eq:pde} with 
		\begin{equation}\label{eq:comp_rn_assumption}
			u(x, t) - v(y, t) \le L(1+|x|+|y|) \text{ for all } (x, y, t) \in \Rn \times \Rn \times I
		\end{equation}
		for some $L > 0$ which is independent of $t$. If $u(\cdot, 0) \le v(\cdot, 0)$ then $u \le v \text{ in } \Rn \times I$.
	\end{theorem}
	\begin{proof}
		We will again start by an exponential rescaling, i.e., take $\lambda > 0$ and define $U(x, t) \coloneqq e^{-\lambda t} u(x, t)$ and $V(x, t) \coloneqq e^{-\lambda t}v(x, t)$. In this case, $U$ and $V$ are sub- and supersolutions to the equation
		\[
		W_t + \lambda W + F(x, t, W, \nabla W, D^2 W) \in \mathcal{S}(e^{\lambda t}(W_t + \lambda W)) G(x, t, W, \nabla W),
		\]
		with appropriately redefined $F$ and $G$ (see the proof of Theorem \ref{thm:comp}). We now proceed in steps.
		
		\emph{Step 1: Growth estimate.} The first step of the proof consists in proving that the difference $U-V$ satisfies the growth estimate 
		\begin{equation}
		\sup_{(x, y, t) \in \Rn \times \Rn \times I} U(x, t) - V(y, t) - 2K|x-y| - \tfrac{\gamma}{T-t} < \infty, \label{eq:comp_rn_growth}
		\end{equation}
		where $K \coloneqq K_{F} + \mathcal{S}_{\textrm{max}} K_G$. Following \cite[Theorem~5.1]{Crandall1992}, we choose a family $\beta_R$ of $C^2(\Rn)$ functions such that
		\begin{enumerate}[label=$\roman*)$]
			\item $\beta_R \ge 0$,
			\item $\liminf_{|x| \to \infty} \frac{\beta_R(x)}{|x|} \ge 2L$,
			\item $|D\beta_R(x)| + |D^2\beta_R(x)| \le C$, for $R \ge 1$, $x\in \Rn$,
			\item $\lim_{R\to \infty} \beta_R(x) = 0$ for $x\in \Rn$,
		\end{enumerate}
		where $C > 0$ is a positive constant. A suitable choice would be a radially symmetric cutoff function, namely $\beta_R = 1$ on $B_R(0):= \{r \in \R \ : \ |r|< R\}$ and $\beta_R = 0$ on $B_{2R}(0)^c$. Let us now define the function
		\[
		\Phi(x, y, t) \coloneqq U(x, t) - V(y, t) - 2K(1+|x-y|^2)^{\frac{1}{2}} - \left( \beta_R(x) - \beta_R(y)\right) - \tfrac{\gamma}{T-t}.
		\]
		Note that condition $ii)$ implies that there is a constant $r(R)$ such that $\beta_R(x) \ge \frac{3}{2}L|x|$ if $|x| > r(R)$. Moreover by \eqref{eq:comp_rn_assumption}, we obtain for $|x|, |y| > r(R)$ the estimate
		\[
		\Phi(x, y, t) \le L(1+|x|+|y|) - 2K - \tfrac{3}{2}L|x| - \tfrac{3}{2}L|y| - \tfrac{\gamma}{T} = L - 2K - \tfrac{\gamma}{T} - \tfrac{1}{2}L(|x|+|y|).
		\]
		Hence, the function has to attain its supremum in a compact subset of $\Rn \times \Rn \times I$.  Let $(\hat{x}, \hat{y}, \hat{t})$ be this maximum. First we consider the case $\Phi(\hat{x}, \hat{y}, \hat{t}) \le 0$. We then have for $R$ big enough
		\begin{align*}
		&U(x, t) - V(y, t) - 2K|x-y| - \tfrac{\gamma}{T-t} \\
		&\le \Phi(\hat{x}, \hat{y}, \hat{t}) + 2K(1+|x-y|^2)^{\frac{1}{2}} - 2K|x-y| + \left( \beta_R(x) - \beta_R(y)\right) \\
		&\le 2K+ \left( \beta_R(x) - \beta_R(y)\right) < \infty,
		\end{align*}
		where the right-hand side can be chosen independently of $R$. Therefore, the asserted inequality \eqref{eq:comp_rn_growth} holds. Secondly assume the other case, i.e., $\Phi(\hat{x}, \hat{y}, \hat{t}) > 0$, which implies that
		\begin{equation}
		2K|\hat{x}-\hat{y}| \le U(\hat{x}, \hat{t}) - V(\hat{y}, \hat{t}) - \tfrac{\gamma}{T-\hat{t}}. \label{eq:comp_rn_bound}
		\end{equation}
		In case $\hat{t} = 0$ then we would get
		\[
		0 < \Phi(\hat{x}, \hat{y}, \hat{t}) \le - 2K(1+|\hat{x}-\hat{y}|^2)^{\frac{1}{2}} - \left( \beta_R(\hat{x}) - \beta_R(\hat{y})\right) - \tfrac{\gamma}{T-\hat{t}} \le 0,
		\]
		which is a contradiction. Hence, the maximum $(\hat{x}, \hat{y}, \hat{t})$ lies inside $\Rn \times \Rn \times (0, T)$, yielding
		\begin{align*}
		(a, p + D\beta_R(\hat{x}), X + D^2\beta_R(\hat{x})) &\in \mathcal{P}^{2, +} U(\hat{x}, \hat{t}),\\
		(b, p - D\beta_R(\hat{y}), -X - D^2\beta_R(\hat{y})) &\in \mathcal{P}^{2, -} V(\hat{y}, \hat{t}),
		\end{align*}
		with $a = b+ \frac{\gamma}{(T-\hat{t})^2}$, $p = 2K \frac{\hat{x}-\hat{y}}{1+ |\hat{x}-\hat{y}|^2}$, and $$X = \frac{2K}{1+|\hat{x}-\hat{y}|^2} \operatorname{Id} - 4K \frac{\hat{x}-\hat{y}}{1+|\hat{x}-\hat{y}|^2} \otimes \frac{\hat{x}-\hat{y}}{1+|\hat{x}-\hat{y}|^2}.$$
		This implies that one can find $\mu \in \mathcal{S}(e^{\lambda t}(a+\lambda U))$ and $\nu \in \mathcal{S}(e^{\lambda t}(b+\lambda V))$ such that
		\begin{align}
		a + \lambda U + F(\hat{x}, \hat{t}, U, p + D\beta_R(\hat{x}), X + D^2\beta_R(\hat{x})) &\le \mu G(\hat{x}, \hat{t}, U, p + D\beta_R(\hat{x})) \label{eq:comp_rn_subsol}, \\
		b + \lambda V + F(\hat{y}, \hat{t}, V, p - D\beta_R(\hat{y}), -X - D^2\beta_R(\hat{y})) &\ge \nu G(\hat{y}, \hat{t}, V, p - D\beta_R(\hat{y})) \label{eq:comp_rn_supersol}.
		\end{align}
		Subtracting \eqref{eq:comp_rn_subsol} from \eqref{eq:comp_rn_supersol} shows that 
		\begin{align*}
		a - b \le&~ \lambda (V-U) \\
		&+ F(\hat y,\hat t, V, p - D\beta_R(\hat{y}), -X - D^2\beta_R(\hat{y}))
		- F(\hat x,\hat t, U, p + D\beta_R(\hat{x}), X + D^2\beta_R(\hat{x})) \\
		&- \nu\left( G(\hat y,\hat  t, V, p - D\beta_R(\hat{y}))
		+ G(\hat x,\hat  t, U, p + D\beta_R(\hat{x}))\right) \\
		&+ (\mu - \nu) G(\hat x,\hat  t, U, p + D\beta_R(\hat{x})).
		\end{align*}
		As $V \le U$ and $b \le a$, we have $\mu \le \nu$ and we can estimate the last term in the right-hand side above by $0$. To treat other terms we use condition \ref{it:F_inc}, \ref{it:F_growth}, \ref{it:G_Lipschitz}, and \ref{it:G_growth}. As $F_2$ and $G_2$ are continuous there is $C(\lambda, p, X, D\beta_R, D^2\beta_R) > 0$ locally bounded such that
		\begin{align*}
		\frac{\gamma}{(T-\hat{t})^2}&\le ~(\lambda-\mathcal{S}_{\textrm{max}}L_G)(V-U) + C(\lambda, p, X, D\beta_R, D^2\beta_R) \\
&\quad+ (K_F + \mathcal{S}_{\textrm{max}}K_G) |\hat x - \hat y|
		\end{align*}
		As $p, X$ are bounded by $K$ and $D\beta_R$ and $D^2 \beta_R$ are bounded independently of $R$ we will introduce a constant $C(\lambda) > 0$ which is independent of $R$ and obtain
		\[
		(\lambda-\mathcal{S}_{\textrm{max}}L_G)(U-V) \le C(\lambda) + K |\hat x - \hat y|.
		\]
		As $U \ge V$ at the maximum point, we can choose $\lambda > C_{F_2} + C_{G_2} + 1$, fixing the constant $C(\lambda)$, to obtain with \eqref{eq:comp_rn_bound}
		\[
		U-V \le C(\lambda) + \tfrac{1}{2} (U-V).
		\]
		Hence, the difference $U-V$ is bounded which implies that
		\[
		\Phi(x, y, t) \le \Phi(\hat{x}, \hat{y}, \hat{t}) \le U(\hat{x}, \hat{t}) - V(\hat{y}, \hat{t}) \le 2C(\lambda).
		\]
		By sending $R \to \infty$ we obtain
		\[
		U(x, t) - V(y, t) - 2K(1+|x-y|^2)^{\frac{1}{2}} - \tfrac{\gamma}{T-t} \le 2C(\lambda)
		\]
		and \eqref{eq:comp_rn_growth} is proved.
		
		\emph{Step 2: Comparison principle.} We proceed by contradiction: Let us assume that
		\[
		\sup_{\substack{x\in \Rn \\ t\in [0, T)}} \left\{ U(x, t) - V(x, t)\right\} \eqqcolon \delta > 0
		\]
		and define
		\[
		M_{\alpha, \epsilon, \gamma } \coloneqq \sup_{\substack{x, y \in \Rn \\ t\in I}} \left\{ U(x, t) - V(y, t) - \tfrac{\alpha}{2}|x-y|^2 - \epsilon(|x|^2+|y|^2) - \tfrac{\gamma}{T-t}\right\}.
		\]
		By \eqref{eq:comp_rn_growth}, $M_{\alpha, \epsilon, \gamma}$ is bounded and we have $M_{\alpha, \epsilon, \gamma} > \delta/2$ for $\gamma, \epsilon$ small enough. Furthermore, we see that $M_{\alpha, \epsilon, \gamma}$ is attained at $(\hat{x}, \hat{y}, \hat{t})$ satisfying
		\begin{align} \label{eq:comp_rn_estimate}
		\tfrac{\alpha}{2}|\hat x-\hat y|^2 + \epsilon(|\hat x|^2+|\hat y|^2) &\le U(\hat x, \hat t) - V(\hat y, \hat t) - \tfrac{\gamma}{T-\hat t}
		\le 2K|\hat{x} - \hat{y}| + C \\
		\nonumber
		&\le \tfrac{\alpha}{4} |\hat{x}-\hat{y}|^2 + \tfrac{4K^2}{\alpha}  + C,
		\end{align}
		for some constant $C = C(\lambda, C_{F_1}, C_{F_2}) > 0$. Hence, the maximum is achieved at some $(\hat{x}, \hat{y}, \hat{t}) \in \Rn \times \Rn \times (0, T)$ and we can again apply the Jensen-Ishii Lemma to obtain that
		\begin{align*}
		(a, \alpha(\hat{x}-\hat{y})+2\epsilon \hat{x}, X + 2\epsilon\operatorname{Id}) &\in \mathcal{P}^{2, +} U(\hat{x}, \hat{t}), \\
		(b, \alpha(\hat{x}-\hat{y})-2\epsilon \hat{y}, Y-2\epsilon\operatorname{Id}) &\in \mathcal{P}^{2, -} V(\hat{y}, \hat{t}) ,
		\end{align*}
		with $a-b = \tfrac{\gamma}{(T-\hat{t})^2}$ and $$-3\alpha \left(\begin{array}{cc}\operatorname{Id} & 0 \\ 0 & \operatorname{Id}\end{array}\right) \le \left(\begin{array}{cc}X & 0 \\ 0 & -Y\end{array}\right)\le 3\alpha \left(\begin{array}{cc}\operatorname{Id} & -\operatorname{Id} \\ -\operatorname{Id} & \operatorname{Id}\end{array}\right).$$
		This implies that $X \le Y$. As $U$ is a subsolution and $V$ is a supersolution, we can find $\mu \in \mathcal{S}(e^{\lambda t}(a+\lambda U))$ and $\nu \in \mathcal{S}(e^{\lambda t}(b+\lambda V))$ such that
		\begin{align}
		&a + \lambda U + F(\hat x,\hat t, U, \alpha(\hat{x}-\hat{y})+2\epsilon\hat x, X+ 2\epsilon\operatorname{Id}) \nonumber\\
&\quad- \mu G(\hat x,\hat  t, U, \alpha(\hat{x}-\hat{y}) + 2\epsilon\hat x) \le 0,\label{eq:comp_rn_proof_sub} \\
		&b + \lambda V + F(\hat y,\hat t, V, \alpha(\hat{x}-\hat{y}) - 2\epsilon\hat y, Y- 2\epsilon\operatorname{Id}) \nonumber\\
&\quad- \nu G(\hat y,\hat  t, V, \alpha(\hat{x}-\hat{y})- 2\epsilon\hat y) \ge 0.\label{eq:comp_rn_proof_super}
		\end{align}
		By subtracting \eqref{eq:comp_rn_proof_sub} from \eqref{eq:comp_rn_proof_super}, we obtain
		\begin{align*}
		a - b \le&~ \lambda (V-U)\\
		&+ F(\hat y,\hat t, V, \alpha(\hat{x}-\hat{y}) - 2\epsilon \hat y, Y- 2\epsilon\operatorname{Id}) - F(\hat x,\hat t, U, \alpha(\hat{x}-\hat{y}) + 2\epsilon \hat x, X + 2\epsilon\operatorname{Id}) \\
		&- \nu G(\hat y,\hat t, V, \alpha(\hat{x}-\hat{y}) - 2\epsilon \hat{y}) + \mu G(\hat x,\hat  t, U, \alpha(\hat{x}-\hat{y}) + 2\epsilon \hat{x}).
		\end{align*}
		By adding and subtracting terms and indicating $\hat{p} \coloneqq \alpha(\hat{x}-\hat{y})$ one has that
		\begin{align}
\label{ineq}		\tfrac{\gamma}{(T-\hat{t})^2} \le&~ \lambda (V-U) \\
		&+ F(\hat y,\hat t, V, \hat{p}- 2\epsilon \hat y, Y -2\epsilon\operatorname{Id}) - F(\hat x,\hat t, V, \hat{p}+ 2\epsilon \hat y, X +2\epsilon\operatorname{Id}) \nonumber\\
		&+F(\hat x,\hat t, V, \hat{p}+ 2\epsilon \hat y, X +2\epsilon\operatorname{Id})- F(\hat x,\hat t, U, \hat{p}+ 2\epsilon \hat x, X+2\epsilon\operatorname{Id}) \nonumber\\
		&- \nu G(\hat y,\hat  t, V, \hat{p}- 2\epsilon \hat y)+ \nu G(\hat x,\hat  t, V, \hat{p}+ 2\epsilon \hat x) \nonumber\\
		&- \nu G(\hat x,\hat  t, V, \hat{p}+ 2\epsilon \hat x)+ \nu G(\hat x,\hat  t, U, \hat{p}+ 2\epsilon \hat x) \nonumber\\
		&- \nu G(\hat x,\hat  t, U, \hat{p}+ 2\epsilon \hat x) + \mu G(\hat x,\hat  t, U, \hat{p}+ 2\epsilon \hat x) \nonumber\\
		\le&~ \lambda (V-U) \nonumber\\&+ F(\hat y,\hat t, V, \hat{p}- 2\epsilon \hat y, Y -2\epsilon\operatorname{Id}) - F(\hat x,\hat t, V, \hat{p}+ 2\epsilon \hat y, X +2\epsilon\operatorname{Id}) + 0 \nonumber\\
		&-\nu G(\hat y,\hat  t, V, \hat{p}- 2\epsilon \hat y)+ \nu G(\hat x,\hat  t, V, \hat{p}+ 2\epsilon \hat x) \nonumber\\
		&+ \mathcal{S}_{\textrm{max}} L_G|U-V|+ (\mu - \nu)G(\hat x,\hat  t, U, \hat{p}+ 2\epsilon \hat x). \nonumber
		\end{align}
		In the second inequality we used \ref{it:F_inc} and \ref{it:G_Lipschitz}, along with $U > V$, $X\le Y$, and $a\ge b$. As $e^{\lambda t}(a+\lambda U) > e^{\lambda t}(b + \lambda V)$, \ref{it:S_mon} implies that $\mu-\nu \le 0$ and by the positivity of $G$ it follows that the last term is negative.

		To treat the other terms in \eqref{ineq} we use \ref{it:F_growth}, \ref{it:G_growth}, and \ref{it:S_bdd} to obtain,
		\begin{align*}
		\tfrac{\gamma}{(T-\hat{t})^2} \le&~ (\lambda-\mathcal{S}_{\textrm{max}}L_G) (V-U) + F_1(\hat y, \hat t, V) - F_1(\hat x, \hat t, V) \\&+F_2(\hat t,\hat p- 2\epsilon \hat y, Y -2\epsilon\operatorname{Id}) - F_2(\hat t, \hat p + 2\epsilon \hat x, X + 2\epsilon\operatorname{Id}) \\
		&+\mathcal{S}_\textrm{max} | G_1(\hat y, \hat  t, V) - G_1(\hat x,\hat  t, V) | +\mathcal{S}_\textrm{max} | G_2(t,\hat p + 2\epsilon \hat y) - G_2(t,\hat p- 2\epsilon \hat y) | \\
		\le&~ (\lambda-\mathcal{S}_{\textrm{max}}L_G) (V-U) + \omega_{F_1}(|\hat x - \hat y|) \\&+ F_2(\hat t,\hat p- 2\epsilon \hat y, Y -2\epsilon\operatorname{Id}) - F_2(\hat t,\hat p + 2\epsilon \hat x, X + 2\epsilon\operatorname{Id}) \\
		&+\mathcal{S}_\textrm{max}\omega_{G_1}(|\hat x - \hat y|) +\mathcal{S}_\textrm{max} | G_2(t,\hat p + 2\epsilon \hat y) - G_2(t,\hat p- 2\epsilon \hat y) |
		\end{align*}
		As equation \eqref{eq:comp_rn_estimate} implies that $\alpha |\hat x - \hat y|^2$ is bounded independently of $\epsilon$, and therefore also $p, q, X$, and $Y$, we can take the limit inferior as $\epsilon \to 0$ of the above  inequality and obtain
		\begin{align*}
		\tfrac{\gamma}{(T-\hat{t})^2} \le&~(\lambda-\mathcal{S}_{\textrm{max}}L_G) (V-U) + \omega_{F_1}(|\hat x - \hat y|) + F_2(\hat t,\hat p, Y) - F_2(\hat t,\hat p, X) \\
		&~+\mathcal{S}_\textrm{max}\omega_{G_1}(|\hat x - \hat y|) +\mathcal{S}_\textrm{max} | G_2(t,\hat p) - G_2(t,\hat p) |.
		\end{align*}
		Note that we used the continuity of $F_2$ and of $G$. Finally, by choosing $\lambda \ge \mathcal{S}_{\textrm{max}}L_G$ we can use \ref{it:F_growth} to reach a contradiction as follows
		\begin{align*}
		0 < \tfrac{\gamma}{T}  &\le\liminf_{\alpha \to \infty}\omega_{F_1}(|\hat x - \hat y|) +\omega_{F_2}(|\hat x - \hat y| + \alpha |\hat x - \hat y|^2)+\mathcal{S}_\textrm{max}\omega_{G_1}(|\hat x - \hat y|)= 0.
		\end{align*}
		Hence, $U \le V$ and therefore $u \le v$.
	\end{proof}

\subsection{Existence of solutions: Perron's Method}
	Now that we have the comparison principles at hand, we can construct solutions by the Perron method \cite{Ishii1987}. We will make use of some properties of the set-valued mapping $\mathcal{S}$ which we record in the following proposition for later reference.

        \begin{proposition}[Properties of $\mathcal{S}$] Let $\mathcal{S}: \R \to \mathcal{P}(\R)$ fulfill \emph{S1)} and \emph{S2)}. We have the following:
\begin{enumerate}[label=\rm C\arabic*)]
		\item\label{it:C_seq} If $a_n \to a$ then any sequence $\mu_n \in \mathcal{S}(a_n)$ has a subsequence $\mu_{n_k}$ such that $\mu_{n_k} \to \mu \in \mathcal{S}(a)$;
		\item\label{it:C_epsDelta} For all $a\in \R$ and for each $\epsilon > 0$ there is a $\delta > 0$ such that for all $b\in\R$ with $|a-b| < \delta$ there are $\mu \in \mathcal{S}(a)$ and $\nu \in \mathcal{S}(b)$ with $|\mu-\nu| <\epsilon$;
\item\label{it:C_epsDelta2}  For all $a\in \R$ the set $\mathcal{S}(a)$ is
 compact.
	\end{enumerate}
\end{proposition}

\begin{proof} 
  Ad C1): Let $a_n \to a$, $\mu_n \in \mathcal{S}(a)$. Since the range of $\mathcal{S}$ is bounded by S2) we can extract a not relabeled subsequence $\mu_n$ such that $\mu_n \to \mu$. Take any $\hat \mu \in \mathcal{S}(\hat a)$ and compute
$$(\mu- \hat \mu)(a - \hat a) = \lim_{n\to \infty} (\mu_n- \hat \mu)(a_n - \hat a)\leq 0.$$
If $\mu \not \in \mathcal{S}(a)$, preserving monotonicity one could properly extend the graph of $-\mathcal{S}(a)$ by adding the point $(a,-\mu)$ to it. This contradicts the maximality of $-\mathcal{S}$.

Ad C2): Assume by contradiction that there exist $a\in \R$ and $\epsilon>0$ such that for all $\delta>0$ there exists $b\in \R$ with $|a-b|<\delta$ such that for all $\mu \in \mathcal{S}(a)$ and $\nu \in \mathcal{S}(b)$ one has 
\begin{equation}\label{mu}|\mu - \nu| \geq \epsilon.
\end{equation} 
Let $\delta =1/n$ and denote by $b_n\in \R$ a point fulfilling the above property. One can hence extract a monotone, not relabeled subsequence $b_n\to a$. Note that $b_n \not = a$ definitely, otherwise one could choose $\mu=\nu$ contradicting \eqref{mu}. Let us assume that $b_n \nearrow a$ (the case $b_n \searrow a$ is analogous). Then \eqref{mu} ensures that 
$$ \sup_{r<a}\mathcal{S}(r) \leq \inf \mathcal{S}(a)-\epsilon.$$
This in turn implies that $(\inf \mathcal{S}(a)-\epsilon, \inf \mathcal{S}(a))\cap \mathcal{S}(\R) = \emptyset$, contradicting the maximality of $-\mathcal S$. Indeed, one could then properly extend the graph of $-\mathcal S$ by adding the segment $\{a\} \times (-\inf \mathcal{S}(a),-\inf \mathcal{S}(a)+\epsilon)$ and preserving monotonicity.

Ad C3): As $\mathcal{S}(a)$ is bounded by S2), one has to check closedness. Assume $\mu_n \in \mathcal{S}(a)$ with $\mu_n \to \mu$. By arguing as in the proof of C1) we have that maximality $-\mathcal{S}$ implies that $\mu \in \mathcal{S}(a)$. 
\end{proof}
	
	The following arguments are an adaptation of the theory from \cite[Pages~22--24]{Crandall1992} to the specific form of equation \eqref{eq:pde}.	
	\begin{lemma}[USC envelopes of subsolutions]\label{lem:sup_subsolutions}
Let $\mathcal{F}$ be a non-empty set of subsolutions of \eqref{eq:pde} and define
		\[
		w(x, t) \coloneqq \sup_{u\in \mathcal{F}} u(x, t).
		\]
		If $w(x, t) < \infty$ then the upper-semicontinuous envelope $w^*$ of $w$, is also a subsolution of \eqref{eq:pde}.
	\end{lemma}
	\begin{proof}
		Let $(x, t) \in \Omega \times I$ and $(a, p, X) \in \mathcal{P}^{2, +}w^*(x, t)$. By definition of the upper-semicontinuous envelope one can find a sequence $(x_n, t_n, u_n) \in \Omega \times I \times \mathcal{F}$ such that $u_n(x_n, t_n) \to w^*(x, t)$ and for any other sequence $(x'_n, t'_n) \to (x, t)$ we have $\limsup_{n\to \infty} u_n(x'_n, t'_n) \le w^*(x, t)$. This implies that there is a sequence $(\hat x_n, \hat t_n) \in \Omega \times I$ and a sequence $(a_n, p_n, X_n) \in \mathcal{P}^{2, +} u_n(\hat x_n,\hat t_n)$ such that
		\[
			(\hat x_n, \hat t_n, u_n(\hat x_n, \hat t_n), a_n, p_n, X_n) \to (x, t, w^*(x, t), a, p, X),
		\]
		see \cite[Proposition~4.3]{Crandall1992}. As $u_n \in \mathcal{F}$ is a subsolution there is a $\mu_n \in \mathcal{S}(a_n)$ such that
		\[
			a_n + F(\hat x_n, \hat t_n, u(\hat x_n, \hat t_n), p_n, X_n) \le \mu_n G(\hat x_n, \hat t_n, u(\hat x_n, \hat t_n), p_n).
		\]
		Passing to a subsequence, by \ref{it:C_seq} we have $\mu_{n_k} \to \mu \in \mathcal{S}(a)$ and the last inequality still holds for this subsequence. Letting $k\to \infty$ implies that
		\[
			a + F(x,t, w^*(x, t), p, X) \le \mu G(x, t, w^*(x, t), p).
		\]
		We proved that for all $(x, t) \in \Omega \times I$ and all $(a, p, X) \in \mathcal{P}^{2, +} w^*(x, t)$ there is a $\mu \in \mathcal{S}(a)$ such that the above inequality holds. This shows that $w^*$ is a subsolution of \eqref{eq:pde}.
	\end{proof}
	
	\begin{lemma}\label{lem:bump_const}
Let $u$ be a subsolution of \eqref{eq:pde}. Assume that $u_*$ is not a supersolution at some point $(\hat x, \hat t)$, i.e., there exist $(a, p, X) \in \mathcal{P}^{2, -} u_*(\hat x, \hat t)$ such that for all $\mu \in \mathcal{S}(a)$ we have
		\begin{equation}
		a + F(\hat x, \hat t, u_*(\hat x, \hat t), p, X) < \mu G(\hat x, \hat t, u_*(\hat x, \hat t), p). \label{eq:strict_ineq}
		\end{equation}
		In this case, for any $\epsilon > 0$ there exists a subsolution $u_\epsilon: \Omega \times I \to \R$ satisfying
		\begin{itemize}
			\item $u_\epsilon(x, t) \ge u(x, t)$,
			\item $\sup(u_\epsilon - u) > 0$,
			\item $u_\epsilon(x, t) = u(x, t)$ for all $(x, t)\in \Omega\times I$ with $|(x, t)-(\hat{x}, \hat{t})|\ge\epsilon$.
		\end{itemize}
	\end{lemma}
	\begin{proof}
		Let $(\hat{x}, \hat{t})$ and $(a, p, X) \in \mathcal{P}^{2, -} u(\hat{x}, \hat{t})$ be such that inequality \eqref{eq:strict_ineq} holds. As $\mathcal{S}(a)$ is compact, see \ref{it:C_epsDelta2}, one can find $\alpha > 0$ such that
		\[
		a + F(\hat x, \hat t, u_*(\hat x, \hat t), p, X) - \mu G(\hat x, \hat t, u_*(\hat x, \hat t), p) \le -\alpha
		\]
		for all $\mu \in \mathcal{S}(a)$. Let us define
		\[
		u_{\delta, \gamma}(x, t) \coloneqq u_*(\hat x, \hat t) + \delta + a(t-\hat{t})+ \left<p, x-\hat x\right>+ \tfrac{1}{2}\left<X(x-\hat x), x-\hat x\right> - \tfrac{\gamma}{2} |(x, t) - (\hat{x}, \hat{t})|^2.
		\]
		As $F$ is continuous we have for $|(x, t) - (\hat{x}, \hat{t})| \to 0$,
		\begin{align*}
		&F( x,  t, u_{\delta, \gamma}( x, t), \nabla u_{\delta, \gamma}(x, t), D^2 u_{\delta,\gamma}(x, t)) \\
&\quad = F( x,  t, u_{\delta, \gamma}( x, t), p - \gamma (x-\hat{x}), X - \gamma \operatorname{Id}) \\
		&\quad = F(\hat x, \hat t, u_*(\hat x, \hat t), p, X) + {\rm{o}}(1).
		\end{align*}
		Analogously, for $G$ we have 
		$$
		G( x,  t, u_{\delta, \gamma}( x, t), \nabla u_{\delta, \gamma}(x, t))= G(\hat x, \hat t, u_*(\hat x, \hat t), p) + {\rm{o}}(1).
		$$
		As $\partial_t u_{\delta, \gamma}( x, t) = a -\gamma (t-\hat{t})$, condition \ref{it:C_epsDelta} implies that there are $\mu_\gamma \in \mathcal{S}(\partial_t u_{\delta, \gamma}( x, t))$ with $\gamma$ small enough such that $\mu_\gamma = \mu + o(1)$. Hence we conclude that if $\delta, \gamma, r$ are small enough then is $u_{\delta, \gamma}$ a subsolution of \eqref{eq:pde} in $B_r(\hat{x}, \hat{t})$. Moreover, since
		\begin{align*}
			u(x, t) &\ge u_*(x, t) \ge u_*(\hat{x}, \hat{t})+ \left<a, t-\hat t\right>+ \left<p, x-\hat x\right>\\
&\quad+ \tfrac{1}{2}\left<X(x-\hat x), x-\hat x\right> + {\rm{o}}(|(x, t)-(\hat x, \hat t)|^2),
		\end{align*}
we can choose $\delta = c(\gamma, r)$ to obtain $u(x, t) > u_{\delta, \gamma}(x, t)$ for $(x, t) \in B_r(\hat{x}, \hat{t}) \setminus B_{\frac{r}{2}}(\hat{x}, \hat{t})$. Therefore the function
		\[
			u_{\gamma}(x, t) \coloneqq \left\{\begin{array}{ll}
			\max\{ u(x, t), u_{\delta, \gamma}(x, t) \} & \text{ in } B_r(\hat{x}, 	\hat{t}), \\
			u(x, t) & \text{ elsewhere,}
			\end{array}\right.
		\]
		is a subsolution by Lemma \ref{lem:sup_subsolutions}. It is clear that $u_\gamma(x, t) \ge u(x, t)$ and that in a neighborhood of $(\hat{x}, \hat{t})$ we have $u_\gamma(x, t) > u(x, t)$. For $\epsilon$ given, by choosing $r, \gamma < \epsilon$ we have that $u_{\gamma}$ satisfies all the required properties.
	\end{proof}
	
	\begin{theorem}[Perron's Method]\label{thm:perron}
		Assume that comparison holds for \eqref{eq:pde}, let $u$ be a subsolution and $v$ be a supersolution of \eqref{eq:pde} with $u_* \le v^*$ on $\partial_P(\Omega \times I)$. Then the function
		\[
			W(x, t) \coloneqq \sup\{ w(x, t) \;|\; w \text{ is a subsolution with } u \le w \le v \}
		\]
		is a viscosity solution of \eqref{eq:pde} with $u_* \le W \le v^*$ on $\partial_P(\Omega \times I)$.
	\end{theorem}
	\begin{proof}
		The standard theory of upper- and lower-semicontinuous envelopes provides
		\[
			u_* \le W_* \le W \le W^* \le v^*
		\]
		which in particular entails that $u_* \le W  \le v^*$ on $\partial_P(\Omega \times I)$. We use Lemma \ref{lem:sup_subsolutions} to conclude that $W^*$ is a also a subsolution and hence $W^* \le v$ by comparison. This however implies that $W = W^*$ and in particular $W$ is a subsolution. 
		
		Now assume that $W_*$ fails to be a supersolution at some point $(\hat{x}, \hat{t})$. Then Lemma \ref{lem:bump_const} provides us with subsolutions $W_\epsilon$ that satisfy $W \le W_\epsilon$ and $u \le W_\epsilon \le v$ on $\partial_P(\Omega \times I)$ by choosing $\epsilon$ small enough. Moreover, we have $u \le W_\epsilon$ and by comparison also $W_\epsilon \le v$. By the maximality of $W$ this implies that $W \ge W_{\epsilon}$. Still, there are points where $W_\epsilon > W$ which is a contradiction. Therefore, $W_*$ is a supersolution and by comparison we have $W \le W_*$ and hence $W_* = W^* = W$. We conclude that $W$ is a viscosity solution.
	\end{proof}
	\subsection{Continuous dependence and regularity on the torus}
	In this section we will prove continuous dependence and some regularity of viscosity solutions on the flat torus. In the following, we say that $u$ {\it is a viscosity subsolution (supersolution) of equation} $E(F_1, \mathcal{S}_1, G_1)$ if $u$ is a subsolution (supersolution) to \eqref{eq:pde} with $F= F_1$, $\mathcal{S} = \mathcal{S}_1$, and $G = G_1$.
	 
Let $L_{G_1}, L_{G_2}$ be the Lipschitz-constants from \ref{it:G_Lipschitz} corresponding to equation $E_1$ and $E_2$ respectively, $\mathcal{S}_\mathrm{max}$ the supremum of $\mathcal{S}$ as in \ref{it:S_bdd}. We will call
		\[
		\lambda(E_1, E_2) \coloneqq \mathcal{S}_\mathrm{max} \min\{  L_{G_1}, L_{G_2} \}
		\]
		the minimal exponential scaling factor of these two equations and the scaled difference between the data is
		\[
		\Lambda(E_1, E_2) \coloneqq \frac{||F_1 - F_2||_\infty+\mathcal{S}_\mathrm{max}||G_1 - G_2||_\infty}{\lambda(E_1, E_2)}.
		\]
	
	\begin{theorem}[Continuous dependence on data]\label{thm:cont_dep}

		Let $u$ be a subsolution of $E_1 = E(F_1, \mathcal{S}, G_1)$ and $v$ a supersolution of $E_2 = E(F_2, \mathcal{S}, G_2)$. Assume that a comparison principle and \emph{\ref{it:F_inc}} hold for $E_2$. Moreover, let \emph{\ref{it:S_bdd}}, \emph{\ref{it:S_mon}}, and \emph{\ref{it:G_Lipschitz}} be satisfied by both equations. We then have for all $(x, t) \in \Omega \times I$
		\begin{align*}
		u(x, t)-v(x, t) \le&~ e^{\lambda(E_1, E_2) t} \left(\max_{(x, t) \in \partial_P ( \Omega \times I )}\{ (u(x, t) - v(x, t))^+ \} +\Lambda(E_1, E_2) \right) \\&- \Lambda(E_1, E_2).
		\end{align*}
	\end{theorem}
	\begin{proof}
		Set $$\Lambda(t) \coloneqq e^{\lambda(E_1, E_2) t} \left(\max_{(x, t) \in \partial_P ( \Omega \times I )}\{ (u(x, t) - v(x, t))^+ \}  +\Lambda(E_1, E_2) \right) - \Lambda(E_1, E_2)$$ and define $\tilde{u} \coloneqq u(x, t) - \Lambda(t)$. We will show that $\tilde{u}$ is also a subsolution to $E_2$. In order to check this, we formally compute 
		\begin{align*}
		\tilde{u}_t + F_2(x, t, \tilde{u}, \nabla \tilde{u}, D^2 \tilde{u}) 
		&= u_t - \Lambda'(t) + F_2(x, t, u - \Lambda(t), \nabla u, D^2 u) \\
		&\le u_t - \Lambda'(t) + F_2(x, t, u, \nabla u, D^2 u)\\
		&\le u_t - \Lambda'(t) + F_1(x, t, u, \nabla u, D^2 u) + ||F_1-F_2||_\infty,
		\end{align*}
		where we used \ref{it:F_inc} in the first inequality. Note that, as $u$ is a viscosity subsolution, there exists $\mu \in \mathcal{S}(u_t)$ such that
		\begin{align*}
		\tilde{u}_t + F_2(x, t, \tilde{u}, \nabla \tilde{u}, D^2 \tilde{u}) 
		&\le \mu G_1(x, t, u, \nabla u) - \Lambda'(t) + ||F_1-F_2||_\infty .
		\end{align*}
		Using \ref{it:S_mon}, we find a $\tilde{\mu} \in \mathcal{S}(\tilde{u}_t)$ with $\tilde{\mu} \ge \mu$. Hence, the positivity of $G$, \ref{it:G_Lipschitz}, and \ref{it:S_bdd} give
		\begin{align*}
		&\tilde{u}_t + F_2(x, t, \tilde{u}, \nabla \tilde{u}, D^2 \tilde{u}) \\
		&\le \tilde{\mu} G_1(x, t, u, \nabla u) - \Lambda'(t)+ ||F_1-F_2||_\infty \\
		&\le \tilde{\mu} G_2(x, t, \tilde{u}, \nabla \tilde{u}) + \lambda(E_1, E_2) \Lambda(t) + \mathcal{S}_\mathrm{max}||G_1 - G_2||_\infty - \Lambda'(t)+ ||F_1-F_2||_\infty \\
		&= \tilde{\mu} G_2(x, t, \tilde{u}, \nabla \tilde{u}).
		\end{align*}
		This shows that $\tilde{u}$ is a viscosity subsolution of $E(F_2, \mathcal{S}, G_2)$. Moreover, we have for $(x, t) \in \partial_P (\Omega \times I)$ that
		\[
		\tilde{u}(x, t) \le u(x, t) - \max_{(x, t) \in \partial_P ( \Omega \times I )}\{ (u(x, t) - v(x, t))^+ \}\le v(x, t)
		\]
		which allows us to use the comparison principle to conclude that $\tilde{u} \le v$ on $\Omega \times I$, and the statement follows.
	\end{proof}

	\begin{lemma}[H\"older-Continuity on the Torus]\label{lem:hold_cont}
		Let $u$ be a viscosity solution to \eqref{eq:pde}, i.e., to $E = E(F, \mathcal{S}, G)$, on $\mathbb{T}^n \times I$. Assume that comparison holds and \emph{\ref{it:F_inc}}, \emph{\ref{it:S_bdd}}, \emph{\ref{it:S_mon}}, and \emph{\ref{it:G_Lipschitz}} are satisfied. Moreover, assume that $F$ and $G$ are $\alpha$-H\"older-continuous for $\alpha \in (0,1] $ in the $(x, t)$-variables, i.e., there are $H_F, H_G > 0$ such that
		\begin{align*}
		| F(x, t, r, p, X) - F(y, s, r, p, X) | &\le H_F (|x-y|^\alpha + |t-s|^\alpha), \\
		| G(x, t, r, p) - G(y, s, r, p) | &\le H_G (|x-y|^\alpha + |t-s|^\alpha),
		\end{align*}
		for all $x, y \in \mathbb{T}^n$, $s, t \in [0, T)$, $r \in\R$, $p\in \Rn$, $X\in \operatorname{Sym}(n)$.
		If $u(\cdot, 0)$ is $\alpha$-H\"older-continuous with constant $H_0 > 0$ and $u$ has $\alpha$-growth, i.e., there exists $M > 0$ such that for all $(x, t) \in \mathbb{T}^n \times I$, we have
		\[
		|u(x, t)- u(x, 0)| \le M t^\alpha,
		\]
		then $u$ is $\alpha$-H\"older-continuous, i.e., for all $x, y\in \mathbb{T}^n$, $t, s \in I$ it holds
		\[
		|u(x, t) - u(y, s)| \le H_u (|x-y|^\alpha + |t-s|^\alpha),
		\]
		with $H_u$ depending on $ H_F, H_G, H_0, \mathcal{S}_\mathrm{max}, L_G$, and $ M$.
	\end{lemma}
	\begin{proof}
		Let $h > 0$ and define $v(x, t) \coloneqq u(x+e, t+h)$. Then, $v$ is a viscosity solution of
		\[
		v_t + F(x+e, t + h, v, \nabla v, D^2 v) \in \mathcal{S}(v_t) G(x+e, t+h, v, \nabla v).
		\]
		We will call this shifted equation $E_{e, h} \coloneqq E(F_{e, h}, \mathcal{S}, G_{e, h})$ and apply Theorem \ref{thm:cont_dep} to $u(x, t)$ and $v(x, t)$ in order to obtain
		\begin{align*}
		|u(x, t) - v(x, t)| &\le e^{\lambda(E, E_{e, h}) t} \left(\max_{x \in \mathbb{T}^n}\{ |u(x, 0) - v(x, 0)| \}  +\Lambda(E, E_{e, h}) \right) \\
		&\quad - \Lambda(E, E_{e, h}).
		\end{align*}
		Note that $\lambda(E, E_{e, h}) = \mathcal{S}_\mathrm{max} L_G$,
		\[
		\Lambda(E, E_{e, h}) = \frac{||F - F_{e, h}||_\infty+\mathcal{S}_\mathrm{max}||G - G_{e, h}||_\infty}{\lambda(E, E_{e, h})} \le \frac{H_F+\mathcal{S}_\mathrm{max}H_G}{\mathcal{S}_\mathrm{max} L_G} (|e|^\alpha + h^\alpha),
		\]
		and finally
		\begin{align*}
		&\max_{x \in \mathbb{T}^n}\{ |u(x, 0) - v(x, 0)| \} \\
		&\le \max_{x \in \mathbb{T}^n}\{ |u(x, 0) - u(x+e, 0)| + |u(x+e, 0) - u(x+e, h)| \} \\
		&\le H_0 |e|^\alpha + M h^\alpha .
		\end{align*}
		Combining these information we obtain 
		\[
		|u(x, t) - v(x, t)| \le H_u \left( |e|^\alpha + h^\alpha \right).
	\qedhere	\]
	\end{proof}

\section{Application to the pinning problem} \label{sec:application}
In this section, we leave the general setting of \eqref{eq:pde} and return to  the specific case of equation \eqref{eq:main}.

	\subsection{Equivalence of viscosity solutions and weak solutions on the torus}\label{ssec:equivalence}
Establishing an energy equality for solutions of \eqref{eq:main} in $\R^n$ asks for a control on the decay at infinity, which cannot be assumed in general. In the particular case when $\varphi$ represents periodically distributed obstacles an energetic formulation (on the torus instead of on $\Rn$) makes sense. In this case we prove that viscosity solutions are weak solutions which implies that they satisfy an energy equality. This suggests that viscosity solutions are an appropriate solution concept to treat the equation in the whole space.
		 
The equivalence of weak and viscosity solutions was first studied by Ishii  in \cite{Ishii1995}. We are going to use similar methods to prove regularity of viscosity solutions. This will eventually lead us to the equivalence result.
	
	Let us start by an approximation of the problem. For $\epsilon> 0$, we consider the equation
	\begin{equation}\label{eq:approx_problem}
	\begin{array}{rcll}
	u_t + \xi_{\epsilon}(u_t)\varphi(x,u(x)) - \Laplace u &=& f(x, t) & \text{in $\mathbb{T}^n\times (0,\infty)$}, \\
	u(\cdot, 0) &=& 0 & \text{on $\mathbb{T}^n \times \{0\}$} 
	\end{array}
	\end{equation}
	where $\xi_\epsilon : \R \to [-1, 1]$ is a smooth, increasing function with $\xi_\epsilon(a_\epsilon) \to \xi_0$, $\xi_0 \in \partial R(a)$ whenever $a_\epsilon \to a$, $\varphi \in C^0(\mathbb{T}^n \times \R) \cap W^{1, \infty}(\mathbb{T}^n \times \R)$, $f \in C^0(\mathbb{T}^n \times I) \cap W^{1, \infty}(\Rn \times I)$, and $I = (0, T)$.
	
	We refer to \eqref{eq:approx_problem} as the {\it approximate problem} and start by proving that it obeys a comparison principle and that one can find unique viscosity solutions.
	\begin{theorem}[Comparison and existence for the approximate problem]\label{thm:approx_exist}
		The approximate problem obeys a comparison principle, i.e., let $u$ and $v$ be viscosity solutions to the approximate problem with $u(\cdot, 0) \le v(\cdot, 0)$ then $u \le v$ on $\mathbb{T}^n \times I$.
		
		Moreover, it admits a unique viscosity solution $u^\epsilon \in C(\mathbb{T}^n \times I) \cap W^{1, \infty}(\mathbb{T}^n \times I)$
with $u^\epsilon(\cdot, 0) = 0$ for all $x\in \mathbb{T}^n$ and 
		\begin{equation}\label{eq:approx_growth}
			|u^\epsilon(x, t)| \le t \left(||f||_\infty + ||\varphi||_\infty\right).
		\end{equation}
		Finally, the Lipschitz-constant of $u^\epsilon$ does not depend on $\epsilon$.
	\end{theorem}

	\begin{proof}
		We can rewrite the approximate problem as $$u_t + F(x, t, D^2 u) = S_\epsilon(u_t)G(x, u) ~~\text{ in } \mathbb{T}^n \times (0, T)$$
		with $F(x, t, X) \coloneqq -\operatorname{tr}(X) - f(x, t)$, $S_\epsilon(a) \coloneqq -\xi_\epsilon(a)$, and $G(x, r) \coloneqq \varphi(x, r)$. Hence, $F$, $G$ and $S_\epsilon$ satisfy all assumptions of Theorem \ref{thm:comp} (see also Remark \ref{rmk:torus}), i.e., comparison holds.
		
		In order to construct a solution by Perron's method, we have prove that a sub- and supersolution exist that satisfy the boundary condition in a strong sense. To this aim, define $\underline{u}(x, t) \coloneqq -t\left(||f||_\infty + ||\varphi||_\infty\right)$ and $\overline{u}(x, t) \coloneqq t \left(||f||_\infty + ||\varphi||_\infty\right)$. These functions are sub- and supersolutions respectively and, by comparison, \eqref{eq:approx_growth} holds for the Perron solution $u^\epsilon$ (see Theorem \ref{thm:perron}).
		
		As the unique solution $u^\epsilon$ satisfies the growth condition \eqref{eq:approx_growth}, we can use Lemma \ref{lem:hold_cont} with $\alpha = 1$ and check that $u^\epsilon$ is indeed Lipschitz-continuous and the Lipschitz-constant $L_{u^\epsilon}$ is independent of $\epsilon$. In fact $L_{u^\epsilon} $ depends only on $||f||_{W^{1, \infty}},||\varphi||_{W^{1, \infty}}$, and $||\xi_\epsilon||_\infty$.
	\end{proof}
	
	\begin{theorem}[Improved regularity for the approximate problem]\label{thm:approx_improv_reg}
		Let $u^\epsilon$ be the unique viscosity solution of the approximate problem. Then, for almost every $t \in I$, $u^\epsilon(\cdot, t) \in W^{2, p}(\mathbb{T}^n) \cap C^{1, \alpha}(\mathbb{T}^n)$ and it holds
		\begin{align*}
		&||u^\epsilon(\cdot, t)||_{W^{2, p}(\mathbb{T}^n)} + ||u^\epsilon(\cdot, t)||_{C^{1, \alpha}(\mathbb{T}^n)} \\ &\le C(n, p, \alpha)\left( ||u^\epsilon||_{W^{1, \infty}(\mathbb{T}^n \times I)} +||\varphi||_{L^\infty(\mathbb{T}^n \times \R)} + ||f||_{L^\infty(\mathbb{T}^n)}\right).
		\end{align*}
	\end{theorem}
	
	\begin{proof}
	From  Theorem \ref{thm:approx_exist} we know that $u^\epsilon$ is Lipschitz-continuous with a constant which does not depend on the Lipschitz-constant of $\xi_\epsilon$. We will now prove that $u^\epsilon$ has improved space regularity. To simplify notation we will write $u = u^\epsilon$ for the rest of this proof. We introduce the sup- and inf-convolutions of $u$, i.e.,
		\begin{align*}
		u^\rho(x, t) &\coloneqq \sup_{(y, s)\in \mathbb{T}^n \times I} \left\{u(y, s) - \tfrac{|(y, s)-(x, t)|^2}{2\rho}\right\}, \\
		u_\rho(x, t) &\coloneqq \inf_{(y, s)\in \mathbb{T}^n \times I} \left\{u(y, s) + \tfrac{|(y, s)-(x, t)|^2}{2\rho}\right\}.
		\end{align*}
		As these functions are semi-convex and semi-concave, respectively, we can apply Alexandroff's theorem to conclude that for almost every pair $(x, t) \in \mathbb{T}^n \times I$, we have
		\begin{align*}		
		(\partial_t {u}^\rho(x, t), \nabla{u}^\rho(x, t), D^2 u^\rho(x, t)) &\in \mathcal{P}^{2, +} u^\rho(x, t),  \\
		(\partial_t u_\rho(x, t), \nabla u_\rho(x, t), D^2  u_\rho(x, t)) &\in \mathcal{P}^{2, -} u_\rho(x, t).
		\end{align*}
		Furthermore, the Lipschitz-continuity of $u$ is carried over to the convolutions with the same constant. The Lipschitz-constant bounds the absolute value of the first order jet elements of the convolutions. Moreover, by the {\it magic property} of sup-, inf-convolutions (see for instance \cite[Chapter~4,~Theorem~7]{Katzourakis2015}), it follows that if
		$(a, p, X) \in \mathcal{P}^{2, +} u^\rho(x, t)$ then $(a, p, X) \in \mathcal{P}^{2, +} u(x + \rho p, t + \rho a)$. Hence, we can compute, that
		\begin{align*}
		& a + \xi_\epsilon(a) \varphi(x, u^\rho(x, t)) - \tr(X) - f(x, t) \\
		&\le \xi_\epsilon(a)\left( \varphi(x, u^\rho(x, t)) - \varphi(x+\rho p, u(x + \rho p, t+\rho a))\right) \\
&\quad- \left( f(x, t) - f(x+\rho p, t+\rho a)\right) \\
		&\le L_\varphi \left( \rho |p| + |u^\rho(x, t) - u(x + \rho p, t+\rho a)| \right) + \rho L_f (|a| + |p|) \\
		&\le \rho \left( L_\varphi L_u + 2L_f L_u \right) + L_\varphi \left(|u^\rho(x, t) - u(x, t)| + |u(x, t) - u(x + \rho p, t+\rho a)|\right) \\
		&\le \rho \left( L_\varphi L_u + 2L_f L_u + 2 L_\varphi L_u^2\right) + L_\varphi |u^\rho(x, t) - u(x, t)| \\
		&\eqqcolon \Psi^\rho(x, t),
		\end{align*}
		with $\Psi^\rho$ uniformly bounded and $\Psi^\rho \to 0$ in $C^0_{loc}(\mathbb{T}^n \times I)$. In this series of inequalities we also used the aforementioned Lipschitz-continuity of the convolution and the bounds on the jet elements. We obtain a similar inequality for the inf-convolutions by changing $\le$ by $\ge$ and $\Psi^\rho$ by a $\Psi_\rho$ which has the same uniform bound and also converges to zero. We hence proved that the sup-, inf-convolutions satisfy
		\begin{align*}
		u^\rho_t + \xi(u^\rho_t) \varphi(x, u^\rho(x, t)) - \Delta u^\rho - f(x, t) &\le \Psi^\rho(x, t), \\
		(u_\rho)_t + \xi((u_\rho)_t) \varphi(x, u_\rho(x, t)) - \Delta u_\rho - f(x, t) &\ge \Psi_\rho(x, t),
		\end{align*}
		in the viscosity sense and due to Alexandroff's theorem also almost everywhere. Therefore, we can take a function $\psi \in C^\infty_c(\mathbb{T}^n \times I)$ with $\psi \ge 0$, multiply the equations and integrate to obtain
		\[
		\int_0^T \int_{\mathbb{T}^n} u^\rho_t \psi +\xi(u^\rho_t) \varphi(x, u^\rho) \psi + \nabla u^\rho \cdot \nabla \psi - f \psi \;\mathrm{d}x\;\mathrm{d}t \le \int_0^T \int_{\mathbb{T}^n} \Psi^\rho \psi \;\mathrm{d}x\;\mathrm{d}t.
		\]
		As the first weak derivative of $u$ exists and $u^\rho \to u$ locally uniformly, it is easy to show that $u^\rho_t \weakto u_t$ and $\nabla u^\rho \weakto \nabla u$ in $L^p$ for all $p < \infty$. Moreover, one can find $\eta \in L^\infty(\mathbb{T}^n \times I)$ such that $\xi(u^\rho_t) \weakstarto \eta$ with $|\eta| \le 1$. Sending $\rho \to 0$ we conclude that
		\[
		\int_0^T \int_{\mathbb{T}^n} u_t \psi +\eta\varphi(x, u) \psi + \nabla u \cdot \nabla \psi - f \psi \;\mathrm{d}x\;\mathrm{d}t \le 0.
		\]
		Analogously for the inf-convolutions, we can find an $\tilde{\eta} \in L^\infty(\mathbb{T}^n \times I)$ such that 
		\[
		\int_0^T \int_{\mathbb{T}^n} u_t \psi +\tilde{\eta}\varphi(x, u) \psi + \nabla u \cdot \nabla \psi - f \psi \;\mathrm{d}x\;\mathrm{d}t \ge 0.
		\]
		Combining both inequalities, we see that for all $\psi \in C^\infty_c(\mathbb{T}^n \times I), \psi \ge 0$ we have
		\begin{align*}
		\int_0^T \int_{\mathbb{T}^n} -u_t \psi -\tilde{\eta}\varphi(x, u) \psi + f \psi \;\mathrm{d}x\;\mathrm{d}t &\le \int_0^T \int_{\mathbb{T}^n} \nabla u \cdot \nabla \psi\;\mathrm{d}x\;\mathrm{d}t \\& \le \int_0^T \int_{\mathbb{T}^n} -u_t \psi -\eta\varphi(x, u) \psi + f \psi \;\mathrm{d}x\;\mathrm{d}t.
		\end{align*}
		By choosing $\psi = \psi^n_1(t) \psi_2(x)$ with $\psi_1^n \in C^\infty_c(I), \psi^n_1 \ge 0$ and $\psi_2 \in C^\infty_c(\mathbb{T}^n), \psi_2 \ge 0$ and letting $\psi^n_1 \to \delta_{t_0}$ in distribution, we have that for almost every $t\in I$ the following inequality holds
		\begin{align*}
		&\int_{\mathbb{T}^n} -u_t(\cdot, t) \psi_2 -\tilde{\eta}\varphi(x, u(\cdot, t)) \psi_2 + f \psi_2 \;\mathrm{d}x \\&\le \int_{\mathbb{T}^n} \nabla u(\cdot, t) \cdot \nabla \psi_2\;\mathrm{d}x \\&\le  \int_{\mathbb{T}^n} -u_t(\cdot, t) \psi_2 -\eta\varphi(x, u(\cdot, t)) \psi_2 + f \psi_2 \;\mathrm{d}x.
		\end{align*}
		Note that $-u_t(\cdot, t) - \tilde{\eta}\varphi(x, u(\cdot, t)) + f, -u_t(\cdot, t) - \eta\varphi(x, u(\cdot, t)) + f \in L^\infty(\mathbb{T}^n)$. This shows us that the weak Laplacian of $u(\cdot, t)$ is a linear and continuous functional from $(W^{1, 2}(\mathbb{T}^n), ||\cdot||_{L^1})$ to $\R$. By extension, the weak Laplacian is an element of $(L^1(\mathbb{T}^n))'$ and therefore we have that $-\Delta u(\cdot, t) \in L^\infty(\mathbb{T}^n)$ with
		\begin{align*}
			||\Delta u(\cdot, t)||_{L^\infty(\mathbb{T}^n)}
			&\le ||u||_{W{1, \infty}(\mathbb{T}^n \times I)} + ||\varphi||_{L^\infty(\mathbb{T}^n \times \R)} + ||f||_{L^\infty(\mathbb{T}^n \times I)}.
		\end{align*}
		At this point standard regularity theory applies. As the torus is a compact manifold, we have for almost every $t \in I$ that $u(\cdot, t) \in W^{2, p}(\mathbb{T}^n) \cap C^{1, \alpha}(\mathbb{T}^n)$ with
		\begin{align*}
			&||u(\cdot, t)||_{W^{2, p}(\mathbb{T}^n)} + ||u(\cdot, t)||_{C^{1, \alpha}(\mathbb{T}^n)} \\
&\quad\le C(n, p, \alpha) (||u||_{L^\infty(\mathbb{T}^n \times I)} + ||\Delta u(\cdot, t)||_{L^\infty(\mathbb{T}^n)} ).
		\qedhere \end{align*}
	\end{proof}
	
	\begin{corollary}[Equivalence of solutions for the approximate problem]
		The unique viscosity solution $u^\epsilon$ of the approximate problem is a strong solution, i.e., we have that
	\begin{align}\label{eq:approx_weak}
	u^\epsilon_t(x, t) + \xi_\epsilon(u^\epsilon_t(x, t)) \varphi(x, u^\epsilon(x, t)) - \Delta u^\epsilon(x, t) = f(x, t)\
		\end{align} 
for almost every $(x,t)\in \mathbb{T}^n \times (0,T)$. Such strong solution is unique.
	\end{corollary}
	\begin{proof}
		By Theorem \ref{thm:approx_improv_reg} for almost every $t\in I$ we have $u^\epsilon(\cdot, t) \in W^{2, p}(\mathbb{T}^n) \cap C^{1, \alpha}(\mathbb{T}^n)$ for every $p \in [1, \infty), \alpha \in [0, 1)$ and in particular the derivative of $u^\epsilon(\cdot, t)$ is absolutely continuous. As an absolutely continuous function $\nabla u^\epsilon(\cdot, t)$ is differentiable almost everywhere. Furthermore, $u^\epsilon$ is also Lipschitz-continuous and by Rademacher's theorem $u^\epsilon_t$ exists for almost every $(x, t) \in \mathbb{T}^n \times I$. Combining these information we see that at almost every point in $(x, t) \in \mathbb{T}^n \times I$ the functions $u^\epsilon_t, \nabla u^\epsilon$, and $D^2 u^\epsilon$ exist. Hence,
		\[
			(u^\epsilon_t(x, t), \nabla u^\epsilon(x, t), D^2 u^\epsilon(x, t)) \in \mathcal{P}^{2, +} u^\epsilon(x, t) \cap \mathcal{P}^{2, -} u^\epsilon(x, t)
		\]
		almost everywhere. Therefore, \eqref{eq:approx_weak} follows. Uniqueness can be checked by a standard Gronwall-Lemma argument.
	\end{proof}

	\begin{theorem}[Equivalence of solutions to the original problem]
		The unique viscosity solution $u \in C^0(\mathbb{T}^n \times I)$ of equation \eqref{eq:main} is a weak solution, i.e., there is a function $\eta \in L^\infty(\mathbb{T}^n \times I)$ such that for all $w \in L^2(I, W^{1, 2}(\mathbb{T}^n))$
		\begin{equation}\nonumber
		\int_0^T \int_{\mathbb{T}^n} u_t w + \eta \varphi(x, u) w + \nabla u \cdot \nabla w \;\mathrm{d}x\;\mathrm{d}s = \int_0^T \int_{\mathbb{T}^n} f w  \;\mathrm{d}x\;\mathrm{d}s
		\end{equation}		
		and $\eta \in \partial R(u_t)$ almost everywhere and $u(\cdot, 0) = 0$.
	\end{theorem}

	\begin{proof}
		For $\epsilon > 0$ given consider the unique viscosity and strong solution $u^\epsilon$ of the approximate problem. This is Lipschitz-continuous with a Lipschitz-constant that is independent of $\epsilon$, see Theorem \ref{thm:approx_exist}. Therefore, we can extract a converging subsequence using Arzel\`a-Ascoli's Theorem. The limit $u \in C^0(\mathbb{T}^n \times I)$ is a viscosity solution of \eqref{eq:main} by Theorem \ref{thm:stability} with $u(\cdot, 0) = 0$. As comparison also holds for the limit equation, this limit is the unique viscosity solution.
		
		On the other hand, for some not relabeled subsequences we also have 
		\begin{align*}
			&u^\epsilon \to u \text{ strongly in } C^0(I, W^{1, 2}(\mathbb{T}^n)), \\
			&u^\epsilon \weakto u \text{ weakly in } L^2(I, W^{2, 2}(\mathbb{T}^n)) \cap W^{1, 2}(I, L^2(\mathbb{T}^n)), \\
			&\xi_\epsilon(u_t^\epsilon) \weakstarto \eta \text{ weakly-* in } L^\infty(I \times \mathbb{T}^n).
		\end{align*}
		Sending $\epsilon \to 0$ in equation \eqref{eq:approx_weak} entails
		\[
		\int_0^T \int_{\mathbb{T}^n} u_t w + \eta \varphi(x, u) w + \nabla u \cdot \nabla w \;\mathrm{d}x\;\mathrm{d}s = \int_0^T \int_{\mathbb{T}^n} f w  \;\mathrm{d}x\;\mathrm{d}s.
		\]
		It remains to prove that $\eta \in \partial R(u_t)$. Let us first compute
		\begin{align*}
			&\limsup_{\epsilon \to 0} \int_{0}^{T} \int_{\mathbb{T}^n} \xi_{\epsilon}(u_t^\epsilon) \varphi(x, u^\epsilon) u^\epsilon_t 
			= \limsup_{\epsilon \to 0} \int_{0}^{T} \int_{\mathbb{T}^n} (-u^\epsilon_t +\Delta u + f) u^\epsilon_t \\
			&= \limsup_{\epsilon \to 0}\left( - \int_{0}^{T} ||u_t^\epsilon||^2_{L^2} - \frac{1}{2}||\nabla u^\epsilon(T)||^2_{L^2} + \frac{1}{2}||\nabla u^\epsilon(0)||^2_{L^2} + \int_{0}^{T} \int_{\mathbb{T}^n} f u^\epsilon_t \right) \\
			&\le\left(- \int_{0}^{T} ||u_t||^2_{L^2} - \frac{1}{2}||\nabla u(T)||^2_{L^2} + \frac{1}{2}||\nabla u(0)||^2_{L^2} + \int_{0}^{T} \int_{\mathbb{T}^n} f u_t \right) \\
			&=\int_{0}^{T} \int_{\mathbb{T}^n} (-u_t +\Delta u + f) u_t = \int_{0}^{T} \int_{\mathbb{T}^n} \eta \varphi(x, u) u_t.
		\end{align*}
		The latter and the lower-semicontinuity of $(u, v) \mapsto R(u) \varphi(\cdot, v)$ entail that
		\begin{align*}
			&\int_0^T\int_{\mathbb{T}^n} \eta \varphi(x, u) (w-u_t) + \int_0^T\int_{\mathbb{T}^n}R(u_t)\varphi(x, u) \\
			&\le \liminf_{\epsilon \to 0} \left(\int_0^T\int_{\mathbb{T}^n} \xi_{\epsilon}(u_t^\epsilon) \varphi(x, u^\epsilon) (w-u^\epsilon_t) + \int_0^T\int_{\mathbb{T}^n}R(u^\epsilon_t)\varphi(x, u^\epsilon)\right) \\
			&\le \liminf_{\epsilon \to 0}\int_0^T\int_{\mathbb{T}^n} R(w)\varphi(x, u^\epsilon) =\int_0^T\int_{\mathbb{T}^n} R(w)\varphi(x, u).
		\end{align*}
		This shows that $\eta \varphi(\cdot, u) \in \partial R(u_t) \varphi(\cdot, u)$ almost everywhere. By possibly redefining $\eta$ whenever $\varphi(\cdot, u) = 0$, we obtain $\eta \in \partial R(u_t)$ almost everywhere.
		
		Therefore, the sequence $u^\epsilon$ converges both to the unique viscosity solution and a weak solution. This shows that the unique viscosity solution is a weak solution.
	\end{proof}

	\begin{remark}
		The Gronwall Lemma can again be used to prove that the weak solutions are unique in the following sense: Let $(u, \eta)$ and $(\tilde{u}, \tilde{\eta})$ be weak solution pairs, then $u = \tilde{u}$ and $\eta \varphi(\cdot, u) = \tilde{\eta} \varphi(\cdot, \tilde{u})$ almost everywhere.
	\end{remark}
	
	\subsection{The pinning result}\label{ssec:pinning}
In line with our modeling considerations from the Introduction, we consider the differential inclusion for $u\colon \R^n\to \R$,
\begin{align} \label{eq:model}
		u_t  - \Laplace u + \varphi(x, u(x)) \partial |u_t| &\ni f(x, t), \\
u(\cdot, 0) &= u_0 \quad \text{on $\R^n$}.  \nonumber
\end{align}
where $\varphi(x,y) =  \eta_\delta \ast \chi_{O}$ with $O = \bigcup_{k\in\N} B_{\rho}(z_k(\omega)) \subset \R^{n+1}$, $\delta>0$, $\rho>0$. That is, $\varphi$ is described by a (smoothed out using a standard mollifier $\eta_\delta$) characteristic function of obstacles localized around centers $z_k(\omega) \in \R^{n+1}$, distributed according to a random process, with realization $\omega \in \hat{\Omega}$, a probability space. In the following we will omit referring to the realization $\omega \in \hat\Omega$ unless necessary.

We first note that solutions to this problem exist and admit a comparison principle.
\begin{proposition}[Existence and comparison] \label{prop:existence}
For $f\in C^0(\R^n\times (0,T))$ uniformly continuous and bounded with $T>0$ and $u_0 \in C^2(\R^n)$ with linear growth there exists a unique viscosity solution to \eqref{eq:model} for any $\omega\in \hat\Omega$. 
\end{proposition}
\begin{proof}
This follows immediately by an application of Theorem \ref{thm:perron}.
\end{proof}

The main pinning result can now be proved by reproducing in this context the argument of \cite[Theorem 2.4]{Dondl_11b}. We have the following.

\begin{theorem}[Pinning] \label{thm:main} Assume that the $z_k$ are given by a $n+1$-dimensional Poisson point process with intensity $\lambda>0$. Then, there exists a deterministic $F^*>0$ such that for $||f||_{L^{\infty}} \le F^*$ and $u_0$ bounded there exist random \emph{stationary} sub- and supersolutions $\overline u \colon \R^n \to \R$, $\underline u\colon \R^n \to \R$ such that any viscosity solution $u$ of \eqref{eq:model} satisfies $\underline u(x) \le u(x,t) \le \overline u(x)$. 
\end{theorem}	
\begin{proof}
By translation of the construction in \cite{Dondl_11b} by $\sup u_0$ we obtain a stationary supersolution noting that $1 \in \partial|s|$ at $s=0$. A subsolution is found analogously, by using the fact that $-1 \in \partial|s|$ at $s=0$ as well.
\end{proof}

\begin{remark}
We note that, by \cite[Theorem~2]{Grimmett:2012ig}, we see that the sub- and supersolutions can be chosen such that they are of bounded expectation at every point $x$. In fact, they even have a finite exponential moment.
\end{remark}

Before closing this discussion, let us note that the
original pinning result in \cite{Dondl_11b} uses a slightly different model for the obstacles, including for example a random strength as well as the possibility of stacking obstacles for nearby points in the Poisson point process in an additive fashion. Due to issues with the requirement of global Lipschitz continuity of $\varphi$, we have chosen an a-priori truncated obstacle field. We argue that the model used here is in fact reasonable for many physical situations, e.g., when considering precipitate hardening as for example discussed in~\cite{1903.07505}.

\section*{Acknowledgement}
US acknowledges support by the Vienna Science and Technology Fund (WWTF)
under project MA14-009 and by the Austrian Science Fund (FWF) under project F\,65. LC acknowledges support from the Fonds National de la Recherche, Luxembourg (AFR Grant 13502370).
	
	\bibliographystyle{plain}

\end{document}